\documentclass[a4paper]{article}
\usepackage{latexsym,bm}
\usepackage{amssymb,amsmath,amsthm}
\usepackage[english,german]{babel}
\usepackage[colorlinks=true]{hyperref}
\usepackage{color}

\allowdisplaybreaks[4]

\newtheorem{theorem}{Theorem}[section]
\newtheorem{lemma}[theorem]{Lemma}
\newtheorem{remark}[theorem]{Remark}
\newtheorem{proposition}[theorem]{Propositon}
\newtheorem{definition}[theorem]{Definition}

\numberwithin{equation}{section}
\hypersetup{linkcolor=blue,urlcolor=red,citecolor=red}
\title{Characterizations of  complete stabilizability\thanks{This work was partially supported by the National Natural Science Foundation of China  under grants 11971022, 11871166 and Fundamental Research Funds for the Central Universities, China University of Geosciences(Wuhan) (CUGSX01).}}
\author{Hanbing Liu\thanks{School of Mathematics and Physics, China University of Geosciences (Wuhan), Wuhan, 430074, China (hanbing272003@aliyun.com)} \and Gengsheng Wang\thanks{Center for Application Mathematics, Tianjin University, Tianjin, 300072, China (wanggs@yeah.net)} \and Yashan Xu\thanks{School of Mathematical Sciences, Fudan University, KLMNS, Shanghai, 200433, China (yashanxu@fudan.edu.cn)} \and Huaiqiang Yu\thanks{School of Mathematics, Tianjin University, Tianjin 300354, China (huaiqiangyu@tju.edu.cn)}}
\date{}
\begin{document}
\selectlanguage{english}
\maketitle
\begin{abstract}
We  present several characterizations, via some weak observability inequalities, of the complete stabilizability for a control system $[A,B]$, i.e., $y'(t)=Ay(t)+Bu(t)$, $t\geq 0$, where $A$ generates a $C_0$-semigroup on a Hilbert space $X$ and $B$ is
a linear and bounded operator from another Hilbert space $U$ to $X$.
 We then extend these characterizations  in two directions: first, the control operator $B$ is unbounded; second, the control system is time-periodic.
 We also  give some
sufficient conditions, from the perspective of the spectral projections, to ensure the  weak observability inequalities.
As applications, we  provide several examples, which are not null controllable, but
can be verified, via the   weak observability inequalities, to be completely stabilizable.
\end{abstract}

{\bf Keywords.} complete stabilizability, weak observability inequality, infinite-dimensional  system
\vskip 5pt
{\bf AMS subject classifications.} 93D15, 93D23, 93B05, 93C25

\section{Introduction}\label{yu-section-1}
\subsection{Control system and notation}
In the  literature on infinite-dimensional linear systems, several concepts of stabilization appear, such as  complete stabilization, exponential stabilization, strong stabilization, polynomial stabilization, and logarithmic stabilization.
This paper mainly studies the complete stabilization for the  control system
 $[A,B]$, i.e.,
   \begin{equation}\label{yu-3-2-1}
    y'(t)=Ay(t)+Bu(t),\;\;\;\;t\geq 0,
\end{equation}
 under the  assumptions:
   \begin{enumerate}
  \item [($H_1$)]  The operator $A$, with its domain $D(A)\subset X$, generates a $C_0$-semigroup
 $S(t)$ ($t\geq 0$)
  on a  Hilbert space $X$.
   \end{enumerate}
   \begin{enumerate}
  \item [($H_2$)] The operator $B$ is a linear and bounded operator from another
   Hilbert space $U$ to  $X$. The Hilbert spaces $X$ and $U$ are identified with their dual spaces respectively.
    \end{enumerate}
    We further  study the complete stabilization for both a system $[A,B]$ (where $B$ is unbounded) and a periodic system $[A(\cdot),B(\cdot)]$. To avoid complex definitions in the introduction, we treat them as extensions in Section 3 of this paper.

    Throughout the paper, the following notations will be used: Given $u\in L^2(\mathbb{R}^+;U)$ and $y_0\in X$,  we write  $y(\cdot; u, y_0)$
for the solution to the system (\ref{yu-3-2-1}) with the  initial condition $y(0)=y_0$; $\mathbb{R}^+:=[0,+\infty)$, $\mathbb{N}:=\{1,2,\ldots\}$ and $\mathbf{N}:=\mathbb{N}\cup \{0\}$;
   Given a Hilbert space $X_1$, we write $\|\cdot\|_{X_1}$ and $\langle\cdot,\cdot\rangle_{X_1}$
   for the norm and the inner product of $X_1$ respectively;
     Given Banach spaces $X_1$ and $X_2$, we write $\mathcal{L}(X_1;X_2)$ for the space of all linear and bounded operators from $X_1$ to $X_2$ and  $\mathcal{L}(X_1):=\mathcal{L}(X_1;X_1)$; Given a linear operator $F$, we use $F^*$ to denote its adjoint operator; We denote by
     $I$ the identity operator on any space;
     Write $\rho(A)$ for the resolvent set of the operator $A$.

\subsection{Aim and motivation}

 Let us first review several
concepts related to the control system  \eqref{yu-3-2-1}:
\begin{enumerate}
\item[$(a_1)$] The system (\ref{yu-3-2-1}) is said to be exponentially stabilizable, if there exists $K\in \mathcal{L}(X;U)$,  $\mu>0$ and $C>0$  such that
    $\|S_K(t)\|_{\mathcal{L}(X)}\leq Ce^{-\mu t}$ for all $t\in\mathbb{R}^+$.  Here, $S_K(t)$ ($t\geq 0$) denotes the semigroup generated by $A+BK$.

\item[$(a_2)$] The system  \eqref{yu-3-2-1} is said to be completely (or rapidly)
stabilizable, if for any $\mu>0$, there exists $K:=K(\mu)\in \mathcal{L}(X;U)$
and $C:=C(\mu)>0$ such that $ \|S_K(t)\|_{\mathcal{L}(X)}\leq Ce^{-\mu t}$
for all $t\in\mathbb{R}^+$.

\item[$(a_3)$] The system  \eqref{yu-3-2-1} is said to be null controllable over $[0,T]$ for some $T>0$, if for any $y_0\in X$, there exists $u\in L^2(0,T;U)$ such that $y(T; u, y_0)=0$.
\end{enumerate}
For these concepts, we have the following known facts:
\begin{enumerate}
\item[$(b_1)$]   In  finite-dimensional settings where $A$, $B$ are matrices,  $(a_2)\Leftrightarrow(a_3)$.
    However, in infinite-dimensional settings,
    $(a_3)\Rightarrow(a_2)\Rightarrow (a_1)$ (see  \cite[Proposition 21]{Trelat-Wang-Xu}), but the reverse may be not true.

\item[$(b_2)$]  The null controllability over $[0,T]$ is equivalent to the following observability inequality: there exists $C:=C(T)>0$ such that
$\|S(T)^*\varphi\|_{X}\leq C\|B^*S(T-\cdot)^*\varphi\|_{L^2(0,T;U)}$ for all $\varphi\in X$ (This inequality can be equivalently written as the ``initial time'' observability inequality for the adjoint equation of (\ref{yu-3-2-1}), that is,  $\|z(0)\|_{X}\leq C\|B^*z(\cdot)\|_{L^2(0,T;U)}$, where $z(\cdot)$ is the solution to the adjoint equation $z'(t)=-A^*z(t), z(T)=\varphi\in X$ (see \cite[Chapter 7, Section 2.2]{Li})).

\item[$(b_3)$] The exponential stabilizability is equivalent to the weak observability of the dual system: there is $\alpha\in(0,1)$, $T>0$ and $C>0$ such that $\|S(T)^*\varphi\|_{X}\leq C\|B^*S(T-\cdot)^*\varphi\|_{L^2(0,T;U)} +\alpha\|\varphi\|_X$ for all $\varphi\in X$.
   (This was proved in  \cite[Theorem 1]{Trelat-Wang-Xu}.)
\end{enumerate}
According to the above facts $(b_1)$-$(b_3)$,  the following question is  natural and interesting:
\begin{itemize}
\item How to
characterize the complete stabilizability  by some kind of
observability inequalities?
\end{itemize}
{\it  The aim of this paper is to answer the above question.}

\subsection{Main results}
   The main  theorem of this paper is as follows:
\begin{theorem}\label{yu-theorem-3-4-1}
    The following statements are equivalent:
\begin{enumerate}
  \item [(i)] The control system (\ref{yu-3-2-1}) is completely stabilizable.
  \item [(ii)] For any $\alpha>0$, there are positive constants $C(\alpha)$ and $D(\alpha)$ such that
\begin{equation}\label{10.22Liu01}
    \|S(T)^*\varphi\|_{X}\leq D(\alpha)\|B^*S(T-\cdot)^*\varphi\|_{L^2(0,T;U)}+C(\alpha) e^{-\alpha T}\|\varphi\|_{X},\;\;\mbox{when}\;\;\varphi\in X\;\;\mbox{and}\;\;T>0.
\end{equation}
  \item [(iii)] There exists $T_0\geq 0$ such that for any $T>T_0$ and $\alpha>0$,
  there are positive constants $C(\alpha)$ (which is independent of $T$) and $D(\alpha,T)$ such that
\begin{equation}\label{yu-6-22-1}
    \|S(T)^*\varphi\|_{X}\leq D(\alpha,T)\|B^*S(T-\cdot)^*\varphi\|_{L^2(0,T;U)}+C(\alpha) e^{-\alpha T}\|\varphi\|_{X},\;\;\mbox{when}\;\;\varphi\in X.
\end{equation}
 \item [(iv)] For each $k\in \mathbb{N}$, there are positive constants $T_k$ and $D(k)$ such that
 \begin{equation}\label{10-23Liu02}
    \|S(T_k)^*\varphi\|_{X}\leq D(k)\|B^*S(T_k-\cdot)^*\varphi\|_{L^2(0,T_k;U)}+ e^{-k T_k}\|\varphi\|_{X},\;\;\mbox{when}\;\;\varphi\in X.
\end{equation}
\end{enumerate}
\end{theorem}
   Some comments on Theorem \ref{yu-theorem-3-4-1} are given.
\begin{enumerate}
\item[$(c_1)$] For each $\alpha>0$, inequality (\ref{10.22Liu01}) is a weak observability inequality. Hence, the complete stabilizability
is characterized by  a family of weak observability inequalities. This essentially differs from the exponential
    stabilizability which corresponds to only one weak observability inequality (see the note $(b_3)$).
    The reason that an  exponential function appears  in \eqref{10.22Liu01} (as well as \eqref{yu-6-22-1}
    and \eqref{10-23Liu02})
     is that the system $[A,B]$ is completely stabilizable if and only if for each $\mu>0$, the system
     $[A+\mu I,B]$  is exponentially stabilizable. This  can be seen from the proof of Theorem \ref{yu-theorem-3-4-1}.

\item[$(c_2)$] The statement $(iv)$ can be understood as  a kind of discretization of the statement $(ii)$.

\item[$(c_3)$] In this paper, we will further extend Theorem \ref{yu-theorem-3-4-1} in two directions: First, the control operator $B$ is  unbounded; Second, the control system  is time-periodic. We give these extensions in Section \ref{yu-sect-3}.
 \item[$(c_4)$]   As applications of Theorem \ref{yu-theorem-3-4-1}, as well as its extensions, some examples
 will be given in Subsection \ref{yu-exmp-4}. These examples present several concrete
  control systems, which are not null controllable, but
 can be proved
  to be completely stabilizable, via the weak observability inequalities in Theorem \ref{yu-theorem-3-4-1}, as well as its extensions.

 \item[$(c_5)$]   We  provide some sufficient conditions, from the perspective of the spectral projections,
    to ensure the weak observability inequalities in Theorem \ref{yu-theorem-3-4-1} (see Subsection \ref{yu-suff-1}).

\end{enumerate}

\subsection{Related works and the novelty of this paper}

 There is a lot of literature on the stabilization of infinite-dimensional systems. We mention
\cite{Curtain, Lagnese, Liu, Pritchard, Triggiani} for time-invariant linear systems with bounded control operators; \cite{Ammari, Badra2, Cur-Weiss, Flandoli-Lasiecka-Triggiani, Guo,  Komornik, Komornik-1997, Krstic, Lasiecka-Triggiani, Lions, Russell, Urquiza, Vest, Weiss-Rebarber} for time-invariant linear systems with unbounded control operators; \cite{Azmi,  Kerschbaum,  Kunisch} for time-varying linear systems with bounded or unbounded control operators;   \cite{Badra, Lunardi, WX2} for time-periodic systems;
\cite{Barbu2, Barbu, Bastin, Coron, Coron2} for nonlinear systems.

About the characterizations of the exponential stabilization for infinite-dimensional linear time-invariant systems, we
would like to mention works \cite{ Badra2, Liu, Trelat-Wang-Xu}:
  A   frequency domain criterion
  on the stabilizability is built up for conservative systems with distributed control  in  \cite{Liu};  A unique continuation type criterion
 on the stabilizability (which is also called Fattorini's criterion)  is established in \cite{Badra2} for  parabolic systems;
 A characterization, via a weak observability inequality, of  the stabilizability is given
 for some infinite dimensional systems in \cite{Trelat-Wang-Xu}.
  About the characterizations of the periodically exponential stabilization for infinite-dimensional  linear time-periodic systems, we mention works
\cite{Badra, Badra3, WX1, WX2, Xu}: Some unique continuation type criterions on
the periodic stabilizability, as well as a characterization, via a weak observability inequality, are presented for some parabolic-like time periodic evolution equations
 in \cite{Badra, Badra3}; A characterization, via a detectability  inequality, is given
 for some linear time-periodic evolution systems in \cite{Xu}.
 Certain geometric and analytic characterizations of the periodic stabilizability
 are provided
 for some linear time-periodic evolution systems
 in \cite{WX1, WX2}.

    We emphasize here the  works  \cite{Badra}, \cite{Badra2} and \cite{Badra3},
  where some characterizations of the  stabilizability/
   the periodic stabilizability, with an arbitrarily given decay rate, were obtained for some parabolic-like evolution equations.  It seems for us that some
   characterizations of the complete stabilizability/the periodically complete stabilizability for those equations can be derived from these works.
   Compared these works with ours, we would like to emphasize what follows: First,
   the results obtained there need the assumption that the generator of the control system has compact resolvent,  while  such assumption is not necessary in our work. (This assumption allows one to  decompose the control system into two sub-systems, one is unstable and of finite-dimension and another is stable and of
   infinite-dimension.)
   Second, the works \cite{Badra}, \cite{Badra2} and \cite{Badra3} concern stabilizability, while ours deals with
   complete stabilizability.

  We now explain the novelty of this work:
  \begin{itemize}
    \item We have not found any  characterization via observability inequalities on the complete/periodic complete stabilizability for time-invariant/time-periodic evolution equations in the literature.  Hence, Theorem \ref{yu-3-2-1}, as well as its extensions obtained in this work, seem to be new results. These results may help us to understand the connections and
        the differences between  stabilizability,  complete stabilizability and  null controllability, from the perspective of observability inequalities.
    \item We are working in a  general framework where the generator
    of the system does not need to have compact resolvents. Our Example 1 in  Subsection 4.2 is under such framework. Indeed, the generator of the system in that example
      has only continuous spectrum, and consequently  does not have compact resolvents.

    \item Though the generators of the systems studied in  \cite{Trelat-Wang-Xu, Xu}  also do not need to have compact resolvents, the authors there did not
        obtain the characterizations on the stabilizability/the periodic stabilizability
         for an arbitrarily given decay rate, like \cite{Badra, Badra2,Badra3}. So the approach to the characterizations on the complete stabilizability obtained in this work does not follow from \cite{Trelat-Wang-Xu, Xu}.

  \end{itemize}

\subsection{The plan of this paper}
The rest of the paper is organized as follows: Section \ref{yu-section-12-2} gives the proof of  Theorem \ref{yu-theorem-3-4-1};
Section \ref{yu-sect-3} presents two extensions of Theorem \ref{yu-theorem-3-4-1}; Section \ref{yu-section-app-4} provides several examples and
 shows
some sufficient conditions ensuring the weak observability inequalities.

\section{The proof of Theorem \ref{yu-theorem-3-4-1}}\label{yu-section-12-2}
The next  lemma is quoted from \cite{Trelat-Wang-Xu} and will play an important role in the proof of
 Theorem \ref{yu-theorem-3-4-1}.

\begin{lemma}\label{10.23Liulem1}
    (\cite[Theorem 1]{Trelat-Wang-Xu}) Let $\mu\geq 0$. Let $S_\mu(t)$ ($t\geq 0$) be the semigroup generated by $A+\mu I$. Then the following  statements are equivalent:
\begin{enumerate}
  \item [(i)] The following  system is exponentially stabilizable:
\begin{equation}\label{yu-3-2-2}
    z'(t)=(A+\mu I)z(t)+Bu(t),\;\;\;\;t\in\mathbb{R}^+.
\end{equation}
    \item [(ii)] There exists $\alpha\in(0,1), T>0$ and $C\geq 0$ such that
\begin{equation}\label{10.23Liu06}
    \|S_\mu(T)^*\varphi\|_{X}\leq C\|B^*S_\mu(T-\cdot)^*\varphi\|_{L^2(0,T;U)}+\alpha \|\varphi\|_{X}\;\;\mbox{for any}\;\;\varphi\in X.
\end{equation}
\end{enumerate}
\end{lemma}

 Now, we are in position to prove Theorem \ref{yu-theorem-3-4-1}.
\begin{proof}[The proof of Theorem \ref{yu-theorem-3-4-1}]
We organize the proof in several steps.
\vskip 5pt

\noindent {\it Step 1. We show $(i)\Rightarrow (ii)$.}
\vskip 5pt

Arbitrarily fix $\alpha>0$. By $(i)$,  there exists  $K:=K(\alpha)\in \mathcal{L}(X;U)$ and $C:=C(\alpha)\geq 1$ such that
\begin{equation}\label{yu-6-22-2}
    \|S_{K}(t)\|_{\mathcal{L}(X)}\leq C e^{-\alpha t}\;\;\mbox{for all}\;\;t\in\mathbb{R}^+.
\end{equation}
     Meanwhile, we arbitrarily fix $y_0\in X$  and set
     \begin{equation}\label{wang2.711.1}
     u_{y_0}(t):=KS_{K}(t)y_0,\;\; t\in \mathbb{R}^+.
     \end{equation}
     Then by \eqref{yu-6-22-2} and \eqref{wang2.711.1}, we have
     \begin{equation}\label{wang2.811.1}
     \|u_{y_0}\|_{L^2(0,T;U)}\leq C(2\alpha)^{-\frac{1}{2}}
     \|K\|_{\mathcal{L}(X;U)}
    \|y_0\|_X\;\;\mbox{for any}\;\;T>0.
     \end{equation}
    Next, we arbitrarily fix $T>0$. By the definitions of $S_K(\cdot)$ and $S(\cdot)$ and by \eqref{wang2.711.1}, we see
     \begin{equation*}
     S_K(T)y_0=S(T)y_0+\int_0^TS(T-t)B u_{y_0}(t)dt,
     \end{equation*}
     which implies
\begin{equation*}
-\langle y_0, S(T)^*\varphi\rangle_{X}=-
\langle S_K(T)y_0,\varphi\rangle_{X}
+\int_0^T\langle u_{y_0}(t),B^*S(T-t)^*\varphi\rangle_{X}dt\;\;\mbox{for any}\;\;\varphi\in X.
\end{equation*}
The above, along with \eqref{yu-6-22-2}  and \eqref{wang2.811.1},  yields that for any $\varphi\in X$,
\begin{equation*}
|\langle y_0, S(T)^*\varphi\rangle_{X}|\leq
C\left(e^{-\alpha T}\|\varphi\|_{X}+\|K\|_{\mathcal{L}(X;U)}(2\alpha)^{-\frac{1}{2}}
    \|B^*S(T-\cdot)^*\varphi\|_{L^2(0,T;U)}\right)\|y_0\|_X.
\end{equation*}
Since $y_0\in X$ and $T>0$ were arbitrarily taken, the above implies that for any $\varphi\in X$
and $T>0$,
\begin{equation}\label{yu-6-22-8}
    \|S(T)^*\varphi\|_{X}\leq C
    \left(e^{-\alpha T}\|\varphi\|_{X}+\|K\|_{\mathcal{L}(X;U)}(2\alpha)^{-\frac{1}{2}}
    \|B^*S(T-\cdot)^*\varphi\|_{L^2(0,T;U)}\right).
\end{equation}
Now, \eqref{10.22Liu01}, with $D(\alpha):=C\|K\|_{\mathcal{L}(X;U)}(2\alpha)^{-\frac{1}{2}}$, follows from \eqref{yu-6-22-8} at once. Since $\alpha>0$ was arbitrarily taken, we obtain $(ii)$.

\vskip 5pt
\noindent {\it Step 2. It is trivial that  $(ii)\Rightarrow (iii)$.}

   \vskip 5pt
   \noindent {\it Step 3. We show  $(iii)\Rightarrow (iv)$.}
   \vskip 5pt
Let $T_0$,  $C(\alpha)$ and $D(\alpha, T)$  be given by  $(iii)$.
      Arbitrarily fix  $k\in \mathbb{N}$.
     Let $T_k$ be such that $T_k>T_0$ and $C(k+1)<e^{T_k}$. Write $D(k):=D(k+1,T_k)$.
     Then, by \eqref{yu-6-22-1} (where $\alpha=k+1$ and $T=T_k$), after some direct computations,
      we get (\ref{10-23Liu02}) with the above $T_k$ and $D(k)$.
     Since $k$ was arbitrarily taken from $\mathbb{N}$, we get $(iv)$.

   \vskip 5pt
   \noindent {\it Step 4. We show  $(iv)\Rightarrow (i)$.}
\vskip 5pt

   Arbitrarily fix  $\mu>0$.  We first show that the system \eqref{yu-3-2-2} is exponentially stabilizable.
   To this end, we  take $k_\mu\in \mathbb{N}$ so that $k_\mu-1\leq \mu< k_\mu$. Then by $(iv)$, we can find  $D(k_\mu)>0$ and $T_{k_\mu}>0$ such that
 \begin{equation}\label{10-26Liu01}
    \|S(T_{k_\mu})^*\varphi\|_{X}\leq D(k_\mu)\|B^*S(T_{k_\mu}-\cdot)^*\varphi\|_{L^2(0,T_{k_\mu};U)}+ e^{-k_\mu T_{k_\mu}}\|\varphi\|_{X}\;\;\mbox{for all}\;\;\varphi\in X.
\end{equation}
Meanwhile, we write $S_\mu(t)$ ($t\geq 0$) for  the  $C_0$-semigroup generated by $A+\mu I$. Then it is clear that
\begin{equation}\label{2.8,12.6}
S_\mu(t)^*=e^{\mu t}S(t)^*\;\;\mbox{for all}\;\;t\geq 0.
\end{equation}
Now, multiplying \eqref{10-26Liu01} by $e^{\mu T_{k_\mu}}$, using \eqref{2.8,12.6}, we have
$$
    \|S_\mu(T_{k_\mu})^*\varphi\|_{X}\leq D(k_\mu)e^{\mu T_{k_\mu}}\|B^*S(T_{k_\mu}-\cdot)^*\varphi\|_{L^2(0,T_{k_\mu};U)}+ e^{-(k_\mu-\mu) T_{k_\mu}}\|\varphi\|_{X}.
$$
    Since the first term in the righthand side of above inequality can be written as
\begin{eqnarray*}
D(k_\mu)e^{\mu T_{k_\mu}}\|B^*S(T_{k_\mu}-\cdot)^*\varphi\|_{L^2(0,T_{k_\mu};U)}
&=&D(k_\mu)e^{\mu T_{k_\mu}}\|e^{-\mu(T_{k_\mu}-\cdot)}B^*S_\mu(T_{k_\mu}-\cdot)^*\varphi\|_{L^2(0,T_{k_\mu};U)}\\
&=& D(k_\mu)\|e^{\mu\cdot}B^*S_\mu(T_{k_\mu}-\cdot)^*\varphi\|_{L^2(0,T_{k_\mu};U)},
\end{eqnarray*}
   and the function $e^{\mu  t}, t\in[0,T_{k_\mu}]$ can be dominated by $e^{\mu T_{k_\mu}}$, we get

\begin{equation}\label{10-26Liu001}
    \|S_\mu(T_{k_\mu})^*\varphi\|_{X}\leq D(k_\mu)e^{\mu T_{k_\mu}}\|B^*S_\mu(T_{k_\mu}-\cdot)^*\varphi\|_{L^2(0,T_{k_\mu};U)}+ e^{-(k_\mu-\mu) T_{k_\mu}}\|\varphi\|_{X}\;\;\mbox{for all}\;\;\varphi\in X.
\end{equation}
Since $e^{-(k_\mu-\mu) T_{k_\mu}}<1$, the above
   \eqref{10-26Liu001} leads to   (\ref{10.23Liu06}) with
   \begin{equation*}
   T=T_{k_\mu}>0,\; \alpha=e^{-(k_\mu-\mu) T_{k_\mu}}\in (0,1),\; C=D(k_\mu)e^{\mu T_{k_\mu}}>0.
   \end{equation*}
Then according to Lemma \ref{10.23Liulem1}, the system (\ref{yu-3-2-2}) is exponentially  stabilizable.

We next claim that the system \eqref{yu-3-2-1} is completely stabilizable. Indeed, since the system (\ref{yu-3-2-2}) is exponentially  stabilizable,  there exists $K:=K(\mu)\in\mathcal{L}(X;U)$ and $C(\mu)>0$ such that
\begin{equation*}
\|S_{\mu,K}(t)\|_{\mathcal{L}(X)}\leq C(\mu)\;\;\mbox{for all}\;\;t\in\mathbb{R}^+,
\end{equation*}
 where $S_{\mu,K}(t)$ $(t\geq 0)$ is the semigroup generated by $A+\mu I+BK$.
   This, together with the fact:
   \begin{equation*}
   S_{\mu,K}(t)=e^{\mu t}S_K(t)\;\;\mbox{for all}\;\;t\in\mathbb{R}^+,
   \end{equation*}
        yields that
   \begin{equation*}
\|S_K(t)\|_{\mathcal{L}(X)}\leq C(\mu)e^{-\mu t}\;\;\mbox{for all}\;\;t\in\mathbb{R}^+.
\end{equation*}
           Since  $\mu>0$ was arbitrarily taken, the above  leads to the complete stabilizability for
           the system
      (\ref{yu-3-2-1}), i.e., $(i)$ is true.
\end{proof}

\section{Extensions}\label{yu-sect-3}
In this section we will extend Theorem \ref{yu-theorem-3-4-1} in two directions: The first one is the case that the control operator $B$ is  unbounded, while the second one is the case that the control system  is time-periodic.
\subsection{The case that  the control operator is unbounded}\label{yu-sec-10-18-1}

\par
    This subsection aims to extend
     Theorem \ref{yu-theorem-3-4-1} to the  control system $[A,B]$, i.e.,
   \begin{equation}\label{yu-3-2-11-29}
    y'(t)=Ay(t)+Bu(t),\;\;\;\;t\geq 0,
\end{equation}
   under the  following assumptions:
\begin{enumerate}
  \item [($\widetilde{H_1}$)]  The operator $A$, with its domain\footnote{
     We define a
      norm on $D(A)$ by: $\|x\|_{D(A)}:=\|(\rho_0I-A)x\|_X$, $x\in D(A)$, where $\rho_0\in\rho(A)$
     is arbitrarily fixed.  Then
     $D(A)$ with this norm is a Hilbert space since $A$ as the generator of a $C_0$-semigroup is closed. It is well known that this norm is equivalent to the classical graph norm $\|x\|'_{D(A)}:=(\|x\|^2_X+\|Ax\|^2_X)^{\frac{1}{2}}$, $x\in D(A)$. The same can be said about any generator of
     a $C_0$-semigroup on $X$.} $D(A)\subset X$, is the generator of
   a $C_0$-semigroup $S(t)$ ($t\geq 0$) on $X$.
\end{enumerate}
\begin{enumerate}
 \item[($\widetilde{H_2}$)] The operator $B$ belongs to $\mathcal{L}(U;X_{-1})$, where $X_{-1}$ is the completion of $X$ with respect to the norm $\|z\|_{-1}:=\|(\rho_0I-A)^{-1}z\|_X$, $z\in X$ (where $\rho_0\in \rho(A)$ is arbitrarily fixed).
\end{enumerate}
     \begin{enumerate}
  \item[($\widetilde{H_3}$)] There exists a time $T>0$ and a constant  $C(T)>0$ such that
\begin{equation}\label{yu-10-4-2}
    \int_0^T\|B^*S(t)^*x\|_U^2dt\leq C(T)\|x\|_X^2\;\;\mbox{for all}\;\;x\in D(A^*).
\end{equation}
     (This condition is called the
    regularity property or the admissibility condition (see, for example, \cite[Chapter 2, Section 2.3]{Coron} or
    \cite{Lasiecka-Triggiani}). Here, we notice that $B^*\in \mathcal{L}(D(A^*);U)$ by $(\widetilde{H_2})$ and $(d_3)$ in Remark \ref{yu-remark-12-1} below.)
  \end{enumerate}
  Given  $y_0\in X$ and $u\in L^2(\mathbb{R}^+; U)$, we write $y(\cdot;u,y_0)$
  for the solution to \eqref{yu-3-2-11-29} with the initial condition $y(0)=y_0$.
\begin{remark}\label{yu-remark-12-1}
Several comments on the above assumptions are given.
\begin{enumerate}

\item[$(d_1)$] The operator $A$ (which belongs to $\mathcal{L}(D(A);X)$) has a unique extension, denoted by
$\widetilde{A}$, in the space $\mathcal{L}(X;X_{-1})$, moreover $(\rho_0I-\widetilde{A})^{-1}\in \mathcal{L}(X_{-1},X)$ (see \cite[Chapter 2, Proposition 2.10.3]{Tucsnak-Weiss}). Hence, $(\widetilde{H_2})$ can be replaced by
  the assumption:
  $(\rho_0I-\widetilde{A})^{-1}B\in \mathcal{L}(U;X)$ (see \cite{Flandoli-Lasiecka-Triggiani, Lasiecka-Triggiani}).

  \item[$(d_2)$] Let
   $\widetilde{S}(t):=(\rho_0I-\widetilde{A})S(t)(\rho_0I-\widetilde{A})^{-1}$ on $X_{-1}$ for any $t\geq 0$. Then $\widetilde{S}(t)$ $(t\geq 0)$ is a $C_0$-semigroup on $X_{-1}$ and  $\widetilde{A}$ is the generator of this semigroup (see \cite[Chapter 2, Proposition 2.10.4]{Tucsnak-Weiss}). We call  $\widetilde{S}(t)$ $(t\geq 0)$ as the extension of $S(t)$ $(t\geq 0)$.

\item[$(d_3)$] The space $D(A^*)$, with the norm $\|z\|_{D(A^*)}:=\|(\overline{\rho_0}I-A^*)z\|_X$, $z\in D(A^*)$, is a Hilbert space and
  $X_{-1}$ is isomorphic to the dual space of $D(A^*)$ (see \cite[Chapter 2, Proposition 2.10.1 and Proposition 2.10.2]{Tucsnak-Weiss}). For convenience, we identify the dual space of $D(A^*)$ with $X_{-1}$. Thus, $X_{-1}$ is the dual space of $D(A^*)$ with respect to the pivot space $X$ (see \cite[Chapter 2, Section 2.9]{Tucsnak-Weiss}).
\item[$(d_4)$] Assumption $(\widetilde{H_3})$ is equivalent to that for any $T>0$, there exists a constant $C(T)>0$ such that (\ref{yu-10-4-2}) holds. (See \cite[Chapter 4, Proposition 4.3.2]{Tucsnak-Weiss}.)
\item[$(d_5)$] We can easily check that when $(\widetilde{H_1})$-$(\widetilde{H_3})$
hold, $B^*$ is
     an admissible observation operator and consequently $B$ is an admissible control operator.
      (See \cite[Chapter 4, Definition 4.3.1]{Tucsnak-Weiss},
         \cite[Chapter 4, Definition 4.2.1]{Tucsnak-Weiss}
     and \cite[Chapter 4, Theorem 4.4.3]{Tucsnak-Weiss}.)
Hence, if $(\widetilde{H_1})$-$(\widetilde{H_3})$ are true,
      then, it follows by \cite[Chapter 4, Proposition 4.2.5]{Tucsnak-Weiss}) that when $y_0\in X$ and $u\in L^2(\mathbb{R}^+; U)$, the equation (\ref{yu-3-2-11-29}) has a unique solution $y(\cdot;u,y_0)$ (in $C([0,+\infty);X)$) which is given by $y(t;u,y_0)=\widetilde{S}(t)y_0
      +\int_0^t\widetilde{S}(t-s)Bu(s)ds$. Moreover, for each $T>0$, there exists $C:=C(T)>0$ such that
       \begin{equation*}\label{Liu-11-5-01}
       \|y(t;u,y_0)\|_X\leq C(\|y_0\|_X+\|u\|_{L^2(0,T; U)}), \;\; t\in[0,T].
       \end{equation*}

\item[$(d_6)$] When $u\in L^2(\mathbb{R}^+;U)$ and $t\geq 0$, we only have  $\int_0^t \widetilde{S}(t-s)Bu(s)ds\in X_{-1}$
        under the assumptions $(\widetilde{H_1})$-$(\widetilde{H_2})$;
        but it holds that $\int_0^t \widetilde{S}(t-s)Bu(s)ds\in X$, if we     further assume $(\widetilde{H_3})$ (see \cite[Chapter 4, Proposition 4.2.2]{Tucsnak-Weiss}).

\item[$(d_7)$]
Some examples  satisfying $(\widetilde{H_1})$-$(\widetilde{H_3})$ are given in \cite{Flandoli-Lasiecka-Triggiani, Lasiecka-Triggiani, Tucsnak-Weiss}.

\end{enumerate}
      \end{remark}

  \par
 Throughout this subsection,  $\widetilde{A}$ and  $\widetilde{S}(t)$ $(t\geq 0)$
  denote respectively the extensions of  $A$ and $S(t)$ $(t\geq 0)$, which are given in $(d_1)$ and $(d_2)$ of Remark \ref{yu-remark-12-1}.
\par
     To state the main results of this subsection, we need
     the following definitions  on the stabilization for the system (\ref{yu-3-2-11-29}):
\begin{definition}\label{yu-definition}
    \begin{enumerate}
      \item [(i)]  The system (\ref{yu-3-2-11-29}) is said to be exponentially stabilizable, if there exists  a $C_0$-semigroup $\Phi(t)$ ($t\geq 0$) on $X$, with  its generator $\Lambda : D(\Lambda)\subset X\to X$, and  $K\in\mathcal{L}(D(\Lambda);U)$ such that
\begin{enumerate}
  \item [$(a)$] $\Lambda x=(\widetilde{A}+BK)x$ for all $x\in D(\Lambda)$;
  \item [$(b)$]  there exists  $\alpha>0$ and  $C>0$ such that
    $\|\Phi(t)\|_{\mathcal{L}(X)}\leq Ce^{-\alpha t}$ for any $t\in\mathbb{R}^+$;
  \item [$(c)$]  there exists  $D>0$
  such that
    $\|K\Phi(\cdot)x\|_{L^2(\mathbb{R}^+;U)}\leq D\|x\|_X$ for all $x\in D(\Lambda)$.
    \end{enumerate}
      \item [(ii)] The system \eqref{yu-3-2-11-29} is said to be completely stabilizable,
      if for any $\alpha>0$,   there exists a $C_0$-semigroup $\Phi_\alpha(t)$ ($t\geq 0$) on $X$, with its generator $\Lambda_\alpha: D(\Lambda_\alpha)\subset X\to X$, and  $K_\alpha\in\mathcal{L}(D(\Lambda_\alpha);U)$
      such that
\begin{enumerate}
\item[($a'$)] $\Lambda_\alpha x=(\widetilde{A}+BK_\alpha)x$ for all $x\in D(\Lambda_\alpha)$;
  \item[($b'$)] there exists  $C(\alpha)>0$ such that
    $\|\Phi_\alpha(t)\|_{\mathcal{L}(X)}\leq C(\alpha)e^{-\alpha t}$ for any $t\in\mathbb{R}^+$;
  \item[($c'$)] there exists  $D(\alpha)>0$
  such that
    $\|K_\alpha \Phi_\alpha(\cdot)x\|_{L^2(\mathbb{R}^+;U)}\leq D(\alpha)\|x\|_X$ for all $x\in D(\Lambda_\alpha)$.
\end{enumerate}
    \end{enumerate}
\end{definition}
\begin{remark}
Definition \ref{yu-definition} is inspired by  \cite{Flandoli-Lasiecka-Triggiani, Lasiecka-Triggiani}, where the authors proved that the solvability of the LQ problem $V(y_0)=\inf_{u\in L^2(\mathbb{R}^+;U)}\int_0^\infty[\|y(t;u,y_0)\|_X^2+\|u(t)\|_U^2]dt$ (i.e.,
$V(y_0)<+\infty$ for all $y_0\in X$)
implies the exponential stabilizability of the system (\ref{yu-3-2-11-29}) in the sense of $(i)$ in  Definition \ref{yu-definition}.
On the other hand, with the aid of Lemma \ref{yu-lemma-10-5-2} below, we obtain the reverse.
Hence, the solvability of the above LQ problem is equivalent to
 the exponential stabilizability of the system (\ref{yu-3-2-11-29}) in the sense of $(i)$ in  Definition \ref{yu-definition}
(see Proposition \ref{corollary-1-3-1} below).

  Besides, it deserves mentioning that  Definition \ref{yu-definition}
can be viewed as
  the dual of the concept of estimatability (see \cite[Definition 2.1]{Ramdani-Tucsnak-Weiss}).

\end{remark}

      The main result in this subsection is as follows:
\begin{theorem}\label{yu-theorem-10-5-1}
    Suppose that  $(\widetilde{H_1})$-$(\widetilde{H_3})$ are true. Then the following  statements  are equivalent:
\begin{enumerate}
  \item [(i)] The  system \eqref{yu-3-2-11-29} is completely stabilizable.
  \item [(ii)] For any $\alpha>0$, there are positive constants $C(\alpha)$ and $D(\alpha)$ such that
\begin{equation}\label{yu-11-30-1}
    \|S(T)^*\varphi\|_{X}\leq D(\alpha)\|B^*S(T-\cdot)^*\varphi\|_{L^2(0,T;U)}+C(\alpha) e^{-\alpha T}\|\varphi\|_{X},\;\;\mbox{when}\;\;\varphi\in D(A^*),\;T>0.
\end{equation}
  \item [(iii)] There exists $T_0\geq 0$ such that for any $T>T_0$ and $\alpha>0$,
  there are positive constants $C(\alpha)$ (which is independent of $T$) and $D(\alpha, T)$ such that
\begin{equation}\label{yu-11-30-2}
    \|S(T)^*\varphi\|_{X}\leq D(\alpha,T)\|B^*S(T-\cdot)^*\varphi\|_{L^2(0,T;U)}+C(\alpha) e^{-\alpha T}\|\varphi\|_{X},\;\;\mbox{when}\;\;\varphi\in D(A^*).
\end{equation}
 \item [(iv)] For each $k\in \mathbb{N}$, there are positive constants $T_k$ and $D(k)$ such that
 \begin{equation}\label{yu-11-30-3}
    \|S(T_k)^*\varphi\|_{X}\leq D(k)\|B^*S(T_k-\cdot)^*\varphi\|_{L^2(0,T_k;U)}+ e^{-k T_k}\|\varphi\|_{X},\;\;\mbox{when}\;\;\varphi\in D(A^*).
\end{equation}
\end{enumerate}
\end{theorem}
\begin{remark}\label{yu-remark-12-1-2}
    It follows from (\ref{yu-10-4-2}) that the operators
\begin{equation*}\label{yu-11-4-1}
\begin{cases}
    (x\in D(A^*))\to ((t\rightarrow B^*S(t)^*x)\in L^2(0,T;U)),\\
    (x\in D(A^*))\to ((t\rightarrow B^*S(T-t)^*x)\in L^2(0,T;U)),
\end{cases}
\end{equation*}
    can be extended in a unique way as linear and bounded operators from $X$ into $L^2(0,T;U)$. If we denote these extensions
    in the same manners, then (\ref{yu-11-30-1}) is equivalent to
    \begin{equation*}
    \|S(T)^*\varphi\|_{X}\leq D(\alpha)\|B^*S(T-\cdot)^*\varphi\|_{L^2(0,T;U)}+C(\alpha) e^{-\alpha T}\|\varphi\|_{X},\;\;\mbox{when}\;\;\varphi\in X,\;T>0.
\end{equation*}
    The same can be said about  (\ref{yu-11-30-2}) and (\ref{yu-11-30-3}).
\end{remark}

Before proving Theorem \ref{yu-theorem-10-5-1},  we give some preliminaries. The first one is about the
LQ problem
\begin{equation}\label{yu-22-1-4-1}
    \textbf{\mbox{(LQ)}}_{y_0}:  \;\;\inf_{u\in L^2(\mathbb{R}^+;U)}J(u;y_0), \;\;y_0\in X,
\end{equation}
     where
\begin{equation}\label{yu-10-4-3}
    J(u;y_0):=\int_0^{\infty}[\|y(t;u,y_0)\|_X^2+\|u(t)\|_{U}^2]dt,\;\; u\in L^2(\mathbb{R}^+;U).
\end{equation}
    Let
    \begin{equation}\label{yu-22-1-3-100}
    \mathcal{U}_{ad}(y_0):=\{u\in L^2(\mathbb{R}^+;U):y(\cdot;u,y_0)\in L^2(\mathbb{R}^+;X)\}.
\end{equation}
    
The following Lemma \ref{yu-lemma-10-5-1} can be found in \cite[Theorem 5.2, Page 40]{Lasiecka-Triggiani} (see also \cite[Theorem 2.2]{Flandoli-Lasiecka-Triggiani} and \cite[Propositions 3.2-3.4]{Weiss-Rebarber}):

\begin{lemma}\label{yu-lemma-10-5-1}
      Assume that $(\widetilde{H_1})$-$(\widetilde{H_3})$ hold.  Suppose that  $\mathcal{U}_{ad}(y_0)\neq \emptyset$
      for any $y_0\in X$. Then for each $y_0\in X$, the problem $\textbf{(LQ)}_{y_0}$ has a unique solution $u_{y_0}^*$. Moreover,  there exists a  self-adjoint and non-negative operator  $P\in \mathcal{L}(X)$ and a $C_0$-semigroup $S_P(t)$ ($t\geq 0$) on $X$, with its generator $A_P:D(A_P)\subset X\to X$,
      such that the following conclusions are true:
\begin{enumerate}
\item[(i)] It holds that $P\in\mathcal{L}(D(A_P);D(A^*))$ and $B^*P\in  \mathcal{L}(D(A_P);U)$.
   \item [(ii)] For each $x\in D(A_P)$, $A_Px=(\widetilde{A}-BB^*P)x$.
\item[(iii)] If $y_0\in D(A_P)$, then $u^*_{y_0}(t)=-B^*P S_P(t)y_0$ for a.e. $t\in\mathbb{R}^+$.
  \item [(iv)] The semigroup $S_P(\cdot)$  is exponentially stable on $X$, i.e.,
there exists $C>0$ and $\alpha>0$  (depending  on $P$) such that
$\|S_P(t)\|_{\mathcal{L}(X)}\leq Ce^{-\alpha t}$ for any $t\in\mathbb{R}^+$.
\end{enumerate}
\end{lemma}

\begin{remark}\label{yu-remark-12-1-3}
 In general, to ensure  the conclusion  $(iv)$ in  Lemma \ref{yu-lemma-10-5-1},  one needs the \emph{detectability condition} given in
     \cite[(D.C), Page 41]{Lasiecka-Triggiani}. Fortunately, this condition
     holds automatically in our setting, since the operator
     $R$ (given in \cite[(D.C), Page 41]{Lasiecka-Triggiani}) is the identity operator currently.
\end{remark}

 The next lemma  is  the extension of
\cite[Proposition 6]{Trelat-Wang-Xu} to the current setting. Since the proof is the same as that of
\cite[Proposition 6]{Trelat-Wang-Xu}, we omit it.

\begin{lemma}\label{yu-lemma-10-5-2}
    Suppose that  $(\widetilde{H_1})$-$(\widetilde{H_3})$ hold. Let $T>0$ and $\alpha>0$. Then the following
    statements  are  equivalent:
\begin{enumerate}
  \item [(i)] The  system \eqref{yu-3-2-11-29} is cost-uniformly $\alpha$-null controllable in time $T>0$, i.e., there exists $C(\alpha,T)\geq0$ such that for each $y_0\in X$, there exists  $u\in L^2(0,T;U)$ such that $\|y(T;u,y_0)\|_X\leq \alpha\|y_0\|_X$ and $\|u\|_{L^2(0,T;U)}\leq C(\alpha,T)\|y_0\|_X$.
  \item [(ii)] There exists  $C(\alpha,T)\geq 0$ such that
\begin{equation*}
    \|S(T)^*\varphi\|_{X}\leq C(\alpha,T)\|B^*S(T-\cdot)^*\varphi\|_{L^2(0,T;U)}+\alpha \|\varphi\|_{X}\;\;\mbox{for any}\;\;\varphi\in D(A^*).
\end{equation*}
    Moreover, $C(\alpha,T)$ in both (i) and (ii) can be taken as the same.
\end{enumerate}
\end{lemma}

\begin{proof}[Proof of Theorem  \ref{yu-theorem-10-5-1}] We organize the proof in several steps.
\vskip 5pt
\noindent \emph{Step 1. We show $(i)\Rightarrow (ii)$.}
\vskip 5pt

Suppose that $(i)$ holds, i.e., $[A,B]$ is completely stabilizable in the sense of
Definition \ref{yu-definition}.
Arbitrarily fix $\alpha>0$.
Then by $(ii)$ in Definition \ref{yu-definition}, there exists a $C_0$-semigroup $\Phi_{\alpha}(t)$ $(t\geq 0)$, with the generator
$\Lambda_{\alpha}:D(\Lambda_{\alpha})\subset X\to X$, and  $K_\alpha\in \mathcal{L}(\Lambda_\alpha);U)$ such that $(a')$-$(c')$ are true.  Several
observations are given in order:
First, by $(a')$ and $(c')$, we have that, for any $y_0\in D(\Lambda_\alpha)$,
\begin{equation*}\label{yu-9-14-1}
   \Phi_\alpha(t)y_0=\widetilde{S}(t)y_0+\int_0^t\widetilde{S}(t-s)BK_\alpha \Phi_\alpha(s)y_0ds,\;\;
    \;t\in\mathbb{R}^+.
\end{equation*}
   Here, we notice that $\Phi_\alpha(t)y_0\in D(\Lambda_\alpha)$ for each $t\in\mathbb{R}^+$ when $y_0\in D(\Lambda_\alpha)$. Second, by $(b')$, we can find $C(\alpha)>0$ such that
$\|\Phi_\alpha(t)\|_{\mathcal{L}(X)}\leq C(\alpha) e^{-\alpha t}$  for all  $t\in\mathbb{R}^+$.
     Third, we let, for each  $y_0\in D(\Lambda_\alpha)$,
     $u_{y_0}(t):=K_\alpha \Phi_\alpha(t)y_0$, $t\in \mathbb{R}^+$. Then
        it follows from $(c')$ of Definition \ref{yu-definition} that there exists $D(\alpha)>0$ (independent of  $y_0$)
     such that
    $\|u_{y_0}(\cdot)\|_{L^2(\mathbb{R}^+;U)}\leq D(\alpha)\|y_0\|_X$, $\;y_0\in D(\Lambda_\alpha)$.

From these observations and by a very similar way as that used in the proof of
$(i)\Rightarrow (ii)$ of Theorem~\ref{yu-theorem-3-4-1}, we can verify that, for each
$y_0\in D(\Lambda_\alpha)$,
\begin{eqnarray*}
&\;&|\langle y_0, S(T)^*\varphi\rangle_{X}|=|\langle S(T)y_0,\varphi\rangle_X|
=|\langle \widetilde{S}(T)y_0,\varphi\rangle_X|
=|\langle \widetilde{S}(T)y_0, \varphi\rangle_{X_{-1},D(A^*)}|\nonumber\\
&=&\left|\left\langle\int_0^T\widetilde{S}(T-s)Bu_{y_0}(s)ds, \varphi\right\rangle_{X_{-1},D(A^*)}-\langle \Phi_\alpha(T)y_0,\varphi\rangle_{X_{-1},D(A^*)}\right|\nonumber\\
&=& \left|\int_0^T\langle u_{y_0}(s),B^*S^*(T-s)\varphi\rangle_Uds-\langle \Phi_\alpha(T)y_0,\varphi\rangle_X\right|\nonumber\\
&\leq&
\left(C(\alpha) e^{-\alpha T}\|\varphi\|_{X}+D(\alpha)
    \|B^*S(T-\cdot)^*\varphi\|_{L^2(0,T;U)}\right)\|y_0\|_X,\;\;\mbox{when}\;\;\varphi\in D(A^*),\; T>0.
\end{eqnarray*}
  The above, along with the density of $D(\Lambda_\alpha)$  in $X$, leads to  \eqref{yu-11-30-1}.
The reason why $D(\Lambda_\alpha)$ is dense in $X$ is that $\Lambda_\alpha$ is the generator of the semigroup
$\Phi_\alpha(t)$ ($t\geq 0$)
(see \cite[Chapter 1, Theorem 1.3]{Pazy}).

\vskip 5pt
   \noindent \emph{Step 2. It is trivial that  $(ii)\Rightarrow (iii)$.}
\vskip 5pt
   \noindent \emph{Step 3.  The proof of $(iii)\Rightarrow (iv)$ is very similar to that used in the proof of  Theorem \ref{yu-theorem-3-4-1}. We omit it.}
\vskip 5pt
    \noindent\emph{Step 4.  We show  $(iv)\Rightarrow (i)$.}
\vskip 5pt
    Arbitrarily fix $\beta>0$ and $y_0\in X$. Let
    $\mathcal{U}^\beta_{ad}(y_0):=\{v\in L^2(\mathbb{R}^+;U):z_\beta(\cdot;v,y_0)\in L^2(\mathbb{R}^+;X)\}$,
    where $z_\beta(\cdot;v,y_0)$ (with $v\in L^2(\mathbb{R}^+;U)$) is the unique solution to the system:
    $z'(t)=(A+\beta I)z(t)+Bv(t)$,\;  $t\in \mathbb{R}^+$;\;\;$ z(0)=y_0$.
    One can directly check that, for each $t\geq 0$,
    $z_\beta(t;v,y_0)=\widetilde{S}_\beta(t)y_0
    +\int_0^t\widetilde{S}_\beta(t-s)Bv(s)ds$. (Here,
    $\widetilde{S}_\beta(t)$ ($t\geq 0$) is the $C_0$-semigroup on $X_{-1}$, generated by $\widetilde{A}+\beta I:X\to X_{-1}$.
    It is easy to check that $\widetilde{S}_\beta(t)=e^{\beta t}\widetilde{S}(t),\;t\geq 0$.)  This, along with the note $(d_4)$ in
     Remark \ref{yu-remark-12-1}, yields that for each $T>0$,  $z_\beta(\cdot;v,y_0)\in C([0,T];X)$.
Consider the  next LQ problem
\begin{equation*}\label{yu-10-6-5}
   \mathbf{(LQ)}^\beta_{y_0}:\ \ \ \  \inf_{v\in L^2(\mathbb{R}^+;U)}\left\{J^\beta(v;y_0)
   :=\int_0^{\infty}[\|z_\beta(t;v,y_0)\|_X^2+\|v(t)\|_{U}^2]dt\right\}.
\end{equation*}
   The rest of the proof of this step is
    divided into  two sub-steps.

    \vskip 5pt

   \noindent \emph{Sub-step 4.1. We prove $\mathcal{U}^\beta_{ad}(y_0)\neq\emptyset$.}

   \vskip 5pt

    Clearly, this will be done, if one can show
     the existence of $\hat{v}\in L^2(\mathbb{R}^+;U)$  such that
\begin{equation}\label{yu-10-6-4}
    \|z_\beta(\cdot;\hat{v},y_0)\|_{L^2(\mathbb{R}^+;X)}\leq \tilde{C}(\beta)\|y_0\|_X\;\;\mbox{and}\;\;
    \|\hat{v}\|_{L^2(\mathbb{R}^+;U)}\leq \tilde{D}(\beta)\|y_0\|_X,
\end{equation}
for some $\tilde{C}(\beta)>0$ and $\tilde{D}(\beta)>0$ depending only on $\beta$.

To show \eqref{yu-10-6-4}, we  construct a control $\hat v$ in the following manner: With respect to the above $\beta>0$,
   there is a unique $k:= k(\beta)\in \mathbb{N}$ satisfying $k-1\leq 2\beta<k$.
Then, according to $(iv)$, there exists $T_k>0$ and $D(k)>0$ such that \eqref{yu-11-30-3} holds. This, together with Lemma \ref{yu-lemma-10-5-2}, implies that the
system \eqref{yu-3-2-11-29} is cost-uniformly $e^{-kT_k}$-null controllable in time $T_k$. Therefore,  there exists  $u_0\in L^2(0,T_k;U)$ such that
\begin{equation}\label{yu-9-26-3}
    \|y(T_k;u_0,y_0)\|_X\leq  e^{-k T_k}\|y_0\|_X\leq  e^{-2\beta T_k}\|y_0\|_X\;\;\mbox{and}\;\;
    \|u_0\|_{L^2(0,T_k;U)}\leq D(k)\|y_0\|_X.
\end{equation}
Let $y_1:=y(T_k;u_0,y_0)$. Then by making use of the above cost-uniformly $e^{-kT_k}$-null controllability again, we can find  $u_1\in L^2(0,T_k;U)$ such that
\begin{equation*}\label{yu-9-26-4}
    \|y(T_k;u_1,y_1)\|_X\leq  e^{-2\beta T_k}\|y_1\|_X\;\;\mbox{and}\;\;\|u_1\|_{L^2(0,T_k;U)}\leq D(k)\|y_1\|_X.
\end{equation*}
    Since the system (\ref{yu-3-2-11-29}) is time-invariant, continuing the above process leads to  a sequence $\{u_i\}_{i\in\mathbf{N}}\subset L^2(0,T_k;U)$ such that
\begin{equation*}\label{yu-9-26-1}
    \|y(T_k;u_i,y_i)\|_X\leq e^{-2\beta T_k}\|y_i\|_X\;\;\mbox{and}\;\;
    \|u_i\|_{L^2(0,T_k;U)}\leq D(k)\|y_i\|_X\;\;\mbox{for any}\;\;i\in\mathbb{N},
\end{equation*}
where $y_i:=y(T_k;u_{i-1},y_{i-1})$.
    This, together with (\ref{yu-9-26-3}),  shows that
\begin{equation}\label{yu-9-26-5}
\|y(T_k;u_i,y_i)\|_X\leq (e^{-2\beta T_k})^{i+1}\|y_0\|_X\leq e^{-2\beta(i+1) T_k}\|y_0\|_X\;\;\mbox{for all}\;\;i\in\mathbf{N};
\end{equation}
  \begin{equation}\label{yu-9-26-6}
    \|u_i\|_{L^2(0,T_k;U)}\leq D(k)(e^{-2\beta T_k})^i\|y_0\|_X
    \leq D(k)e^{-2\beta iT_k}\|y_0\|_X\;\;\mbox{for all}\;\;i\in\mathbf{N}.
\end{equation}
Let
\begin{equation}\label{yu-9-26-10}
    \hat{u}(t):=\sum_{i=0}^\infty \chi_{[iT_k, (i+1)T_k)}(t)u_i(t-iT_k),\;\; t\in\mathbb{R}^+
\end{equation}
    and
\begin{equation}\label{yu-9-26-12}
    z_\beta(t):=e^{\beta t}y(t;\hat{u},y_0),\;\; \hat{v}(t):=e^{\beta t}\hat{u}(t),\;\;t\in\mathbb{R}^+.
\end{equation}
\par
    Now, we show that $\hat v$, given in \eqref{yu-9-26-12}, satisfies the second inequality in (\ref{yu-10-6-4}).
Indeed,
 by (\ref{yu-9-26-6}),  (\ref{yu-9-26-10}) and the second equality in (\ref{yu-9-26-12}), we find
\begin{eqnarray*}
    \|\hat{v}\|_{L^2(\mathbb{R}^+;U)}
    \leq \sum_{i=0}^\infty  e^{\beta (i+1)T_k}\|u_i\|_{L^2(0,T_k;U)}
    \leq D(k) \sum_{i=0}^\infty  e^{-\beta (i-1)T_k}\|y_0\|_X
     = \frac{ D(k)e^{\beta T_k}}{1-e^{-\beta T_k}}\|y_0\|_X,
\end{eqnarray*}
which leads to  the second inequality in (\ref{yu-10-6-4})  with $\tilde{D}(\beta)=\frac{ D(k)e^{\beta T_k}}{1-e^{-\beta T_k}}$. (Here, we notice that  $k$ is uniquely determined by $\beta$, and thus the constant $\frac{ D(k)e^{\beta T_k}}{1-e^{-\beta T_k}}$ depends only on $\beta$.)

Finally, we show that $\hat v$, given by \eqref{yu-9-26-12}, satisfies the first inequality in (\ref{yu-10-6-4}). To this end, several observations are given in order. First, it follows from  \eqref{yu-9-26-12} that
   $z_\beta(t)=z_\beta(t;\hat v,y_0)$ for any $t\geq 0$.
   Then, it follows by the equation satisfied by $z_\beta(\cdot;\hat v,y_0)$,
     (\ref{yu-9-26-12}) and (\ref{yu-9-26-10}) that  when $t\in[iT_k,(i+1)T_k]$,
   with $i\in\mathbf{N}$ arbitrarily fixed,
\begin{eqnarray}\label{yu-9-26-15}
    &\;&|\langle z_\beta(t;\hat{v},y_0), \varphi\rangle_X|=|\langle z_\beta(t;\hat{v},y_0), \varphi\rangle_{X_{-1},D(A^*)}|\nonumber\\
    &\leq&|\langle \widetilde{S}_\beta(t-iT_k)z_{\beta}(iT_k;\hat{v},y_0),\varphi\rangle_{X_{-1},D(A^*)}|
    +\left|\left\langle\int_{iT_k}^t\widetilde{S}_\beta(t-s)B\hat{v}(s)ds,\varphi\right
    \rangle_{X_{-1},D(A^*)}\right|
    \nonumber\\
    &=&\left|\langle e^{\beta(t-iT_k)}S(t-iT_k)z_\beta(iT_k;\hat{v},y_0),\varphi\rangle_X\right|
    +\left|\int_{iT_k}^t\langle\widetilde{S}_\beta(t-s)B\hat{v}(s),\varphi
    \rangle_{X_{-1},D(A^*)}ds\right|\nonumber\\
    &\leq&\sup_{t\in[0,T_k]}\|S(t)\|_{\mathcal{L}(X)}e^{\beta T_k}
\|z_\beta(iT_k;\hat{v},y_0)\|_X\|\varphi\|_X\nonumber\\
&\;&+e^{\beta (i+1)T_k}\left|\int_{iT_k}^t\langle\widetilde{S}(t-s)B\hat{u}(s),\varphi\rangle_{X_{-1},D(A^*)}
ds\right| \nonumber\\
         &\leq& \sup_{t\in[0,T_k]}\|S(t)\|_{\mathcal{L}(X)}e^{\beta (i+1)T_k}\|y(iT_k;\hat{u},y_0)\|_X\|\varphi\|_X\nonumber\\
    &&+e^{\beta (i+1)T_k}\left|\int_{0}^{t-iT_k}\langle\widetilde{S}(t-iT_k-s)Bu_i(s),
    \varphi\rangle_{X_{-1},D(A^*)}ds\right|\;\;\mbox{for any}\;\;\varphi\in D(A^*).
\end{eqnarray}
   Here, we used the fact $\widetilde{S}(t)=S(t)$ $(t\geq 0)$ on $X$. Second,  we get from  \eqref{yu-9-26-10} and the construction of $\{y_i\}_{i\in\mathbf{N}}$ that
for each $i\in \mathbb{N}$,
$y(iT_k;\hat{u},y_0)=y(T_k; u_{i-1}, y_{i-1})$.
This, along with  \eqref{yu-9-26-5}, leads to
\begin{eqnarray}\label{Liu-11-11-1}
    \|y(iT_k;\hat{u},y_0)\|_X\leq e^{-2\beta iT_k}\|y_0\|_X,\;\;\mbox{when}\;\;i\in \mathbb{N}.
\end{eqnarray}
Third, it follows  by \eqref{yu-10-4-2} and \eqref{yu-9-26-6} that for each $t\in [iT_k,(i+1)T_k]$
(with $i\in\mathbf{N}$ arbitrarily fixed) and for each $\varphi\in D(A^*)$,
\begin{eqnarray}\label{Liu-11-11-2}
    \left|\int_0^{t-iT_k}\langle \widetilde{S}(t-iT_k-s)Bu_i(s),\varphi\rangle_{X_{-1},D(A^*)}ds\right|=
    \left|\int_0^{t-iT_k}\langle u_i(s),B^*S^*(t-iT_k-s)\varphi\rangle_Uds\right|\nonumber\\
    \leq C(T_k)\|u_i\|_{L^2(0,T_k;U)}\|\varphi\|_X\leq C(T_k)D(k)e^{-2\beta iT_k}\|y_0\|_X\|\varphi\|_X.
\end{eqnarray}
 Since $D(A^*)$ is dense in $X$, it follows \eqref{yu-9-26-15}, \eqref{Liu-11-11-1} and \eqref{Liu-11-11-2} that
\begin{eqnarray*}
   \|z_\beta(t;\hat{v},y_0)\|_{L^2(\mathbb{R}^+;X)}&\leq&\sum_{i=0}^\infty \|z_\beta(\cdot;\hat{v},y_0)\|_{L^2(iT_k,(i+1)T_k;X)}\leq T_k\sum_{i=0}^\infty \sup_{t\in[iT_k,(i+1)T_k]}\|z_\beta(t;\hat{v},y_0)\|_X\nonumber\\
    &\leq& T_ke^{\beta T_k}\left(\sup_{t\in[0,T_k]}\|S(t)\|_{\mathcal{L}(X)}+C(T_k)D(k)\right)\sum_{i=0}^\infty e^{-\beta iT_k}\|y_0\|_X,
\end{eqnarray*}
which leads to  the first inequality in (\ref{yu-10-6-4}) with $\tilde{C}(\beta):=T_ke^{\beta T_k}(\sup_{t\in[0,T_k]}\|S(t)\|_{\mathcal{L}(X)}+C(T_k)D(k))/(1-e^{-\beta T_k})$.
 The reason why $D(A^*)$ is dense in $X$ is that $A^*$ is the generator of the adjoint  semigroup $S(t)^*$ $(t\geq 0)$ (see \cite[Chapter 1, Corollary 10.6]{Pazy})).

\vskip 5pt
\noindent \emph{Sub-step 4.2. We prove the desired complete stabilizability.}
\vskip 5pt

One can directly check that $[A+\beta I,B]$ still satisfies the assumptions
  $(\widetilde{H_1})$-$(\widetilde{H_3})$. Meanwhile, by Sub-step 4.1, we have $\mathcal{U}^\beta_{ad}(y_0)\neq \emptyset$
  for each $y_0\in X$.
  Thus,
  by Lemma \ref{yu-lemma-10-5-1}
    (where  $\mathbf{(LQ)}_{y_0}$ is replaced by  $\mathbf{(LQ)}_{y_0}^\beta$), there is    a unique solution $v_{y_0}^*$ to $\mathbf{(LQ)}_{y_0}^\beta$; a
    self-adjoint and non-negative definite operator $P:=P(\beta)\in \mathcal{L}(X)$;  a $C_0$-semigroup $S_P^\beta(t)$ $(t\geq 0)$ on $X$, with the generator $A_P^\beta: D(A_P^\beta)\subset X\to X$,
    such that  the following conclusions are true:
\begin{enumerate}
  \item [$(e_1)$]  It holds that $P\in\mathcal{L}(D(A^\beta_P);D(A^*))$ and $B^*P\in  \mathcal{L}(D(A_P^\beta);U)$;
\item [$(e_2)$] For any $y_0\in D(A_P^\beta)$, $A_P^\beta y_0=(\widetilde{A}+\beta I-BB^*P)y_0$;
  \item[$(e_3)$] If $y_0\in D(A_P^\beta)$, then $v_{y_0}^*(t)=-B^*PS_P^\beta(t)y_0$ for a.e. $t>0$;
  \item [$(e_4)$] The semigroup $S_P^\beta(t)$\; $(t\geq 0)$ is exponentially stable on $X$.
\end{enumerate}
\par
    Let $K_\beta:=-B^*P$  and $\Phi_\beta(t):=e^{-\beta t}S_P^\beta(t),\; t\geq 0$.
    Then one can directly check  that $\Phi_\beta(t)$ $(t\geq 0)$ is a $C_0$-semigroup on $X$ generated by $\Lambda_\beta:=A_P^\beta-\beta I$,
     with $D(\Lambda_\beta)=D(A_P^\beta)$. Moreover, by $(e_1)$,  $K_\beta\in\mathcal{L}(D(\Lambda_\beta);U)$.
\par
 Next, we will check that the above $K_\beta$ and $\Phi_\beta(t)$ ($t\geq 0$) satisfy $(a')$-$(c')$ in Definition \ref{yu-definition} one by one.
 First, it follows from  $(e_2)$
       that   $\Lambda_\beta y_0=(A_P^\beta-\beta I)y_0=(\widetilde{A}-BB^*P)y_0$  for any $y_0\in D(\Lambda_\beta)(=D(A_P^\beta))$,
       which leads to
       $(a')$ in Definition \ref{yu-definition}. Second,
       we use  $(e_4)$ to find $C:=C(\beta)>0$ such that $\|S_P^\beta(t)\|_{\mathcal{L}(X)}\leq C$ for all
       $t>0$. Thus we have
      \begin{equation*}\label{Liu-11-22-01}
 \|\Phi_\beta(t)\|_{\mathcal{L}(X)}=e^{-\beta t}\|S_P^\beta(t)\|_{\mathcal{L}(X)}\leq Ce^{-\beta t}, \;\;\mbox{when}\;\; t>0,
\end{equation*}
which leads to $(b')$  in Definition \ref{yu-definition}.
Finally, we use  \eqref{yu-10-6-4}
to find $\hat{D}(\beta)>0$ such that
\begin{equation*}\label{Liu-11-12-01}
    \|v^*_{y_0}\|_{L^2(\mathbb{R}^+; U)}\leq \sqrt{J^\beta(v^*_{y_0};y_0)}=\sqrt{\inf_{v\in L^2(\mathbb{R}^+;U)}J^\beta(v;y_0)}\leq \hat{D}(\beta)\|y_0\|_{X}.
\end{equation*}
This, together with $(e_3)$, implies that when $y_0\in D(\Lambda_\beta)(=D(A_P^\beta))$,
\begin{equation*}
  \|K_\beta \Phi_\beta(\cdot)y_0\|_{L^2(\mathbb{R}^+; U)}=\|-e^{-\beta \cdot}B^*PS_P^\beta(\cdot)y_0\|_{L^2(\mathbb{R}^+; U)}\leq \|v^*_{y_0}\|_{L^2(\mathbb{R}^+; U)}\leq \hat{D}(\beta)\|y_0\|_{X},
  \end{equation*}
which leads to  $(c')$  in Definition \ref{yu-definition}.

Now, since $\beta>0$ was arbitrarily taken, the above checked $(a')$-$(c')$, along with Definition \ref{yu-definition}, yields that the system \eqref{yu-3-2-11-29}
is completely stabilizable.
\end{proof}

    The following proposition may have independent interest.
\begin{proposition}\label{corollary-1-3-1}
     Suppose that  $(\widetilde{H_1})$-$(\widetilde{H_3})$ hold. Then the following  statements  are equivalent:
\begin{enumerate}
  \item [$(i)$] The set $\mathcal{U}_{ad}(y_0)$ defined by (\ref{yu-22-1-3-100}) is nonempty for any $y_0\in X$.
  \item [$(ii)$] The system (\ref{yu-3-2-11-29}) is exponentially stabilizable (in the sense of
    $(i)$ in Definition \ref{yu-definition}).

   \item [$(iii)$]  For any $y_0\in X$, $V(y_0):=\inf_{u\in L^2(\mathbb{R}^+;U)}J(u;y_0)<+\infty$,
   where $J(u;y_0)$ is given by \eqref{yu-10-4-3}.

\end{enumerate}
\end{proposition}
\begin{proof}
First of all, it is well known that $(i)\Leftrightarrow (iii)$.

    We next prove $(i)\Rightarrow (ii)$. Suppose $(i)$ holds. Then it follows from Lemma
    \ref{yu-lemma-10-5-1} that for each $y_0\in X$, the problem $\textbf{(LQ)}_{y_0}$ (see \eqref{yu-22-1-4-1}) has a unique solution $u_{y_0}^*$,  moreover  there exists a  self-adjoint and non-negative operator  $P\in \mathcal{L}(X)$ and a $C_0$-semigroup $S_P(t)$ ($t\geq 0$) on $X$, with its generator $A_P:D(A_P)\subset X\to X$,
      such that $(i)$-$(iv)$ in Lemma \ref{yu-lemma-10-5-1} are true. Let
      $\Phi(t):=S_P(t)$ ($t\geq 0$). (Its the generator  is
$\Lambda:=A_P$.) Let $K:=-B^*P$. Then by $(i)$ in Lemma \ref{yu-lemma-10-5-1}, we have
$K\in \mathcal{L}(D(\Lambda);U)(=\mathcal{L}(D(A_P);U))$.

 Now we  show that the above $\Phi(t)$ and $K$ satisfy  the conditions $(a)$, $(b)$ and $(c)$ in  Definition \ref{yu-definition}.
 Indeed,  $(a)$ and $(b)$ follow from  $(ii)$ and $(iv)$ in Lemma \ref{yu-lemma-10-5-1}, respectively.
 While the condition $(c)$ can be deduced from our assumptions and $(iii)$ in Lemma \ref{yu-lemma-10-5-1}.
    Indeed, by our assumptions and $(ii)$ in \cite[Theorem 2.2]{Flandoli-Lasiecka-Triggiani}, there exists a constant $C>0$ such that
\begin{eqnarray*}
    \|u^*_x\|_{L^2(\mathbb{R}^+;U)}
    \leq \sqrt{\inf_{u\in L^2(\mathbb{R}^+;U)}J(u;x)}\leq C\|x\|_X,\;\;\mbox{when}\;\;x\in X,
\end{eqnarray*}
    where  $J(u;x)$ is defined by (\ref{yu-10-4-3}) (with $y_0=x$).  Thus, by $(iii)$ in Lemma \ref{yu-lemma-10-5-1}, we get $(c)$. Hence
    $(ii)$ is true.
\par
    Finally, we show $(ii)\Rightarrow (i)$. We suppose $(ii)$ holds, i.e., there exists  a $C_0$-semigroup $\Phi(t)$ ($t\geq 0$) on $X$, with  its generator $\Lambda : D(\Lambda)\subset X\to X$, and  $K\in\mathcal{L}(D(\Lambda);U)$ such that
    the conditions $(a)$, $(b)$ and $(c)$ are true. Let $\alpha>0$ be given in $(b)$.
    By a very similar way used in  Step 1 in the proof of Theorem  \ref{yu-theorem-10-5-1}, we can find positive constants $C(\alpha)$ and $D(\alpha)$ such that
    (\ref{yu-11-30-1}) holds for the aforementioned $\alpha$. Thus, there exists
    $T_0>0$ and  $\delta\in(0,1)$ such that
\begin{equation*}
    \|S(T_0)^*\varphi\|_{X}\leq D(\alpha)\|B^*S(T_0-\cdot)^*\varphi\|_{L^2(0,T_0;U)}+\delta\|\varphi\|_{X},\;\;\mbox{when}\;\;\varphi\in D(A^*).
\end{equation*}
      This, together with Lemma
     \ref{yu-lemma-10-5-2}, yields that the  system \eqref{yu-3-2-11-29} is cost-uniformly $\delta$-null controllable at time $T_0>0$. Then, by the very similar way used in Sub-step 4.1
     in the proof of Theorem  \ref{yu-theorem-10-5-1} (or in the proof of Lemma 31 in \cite{Trelat-Wang-Xu}), we can conclude that
     $\mathcal{U}_{ad}(y_0)\neq \emptyset$ for any $y_0\in X$, i.e.,  $(i)$ is true.
\end{proof}

\subsection{Periodic feedback stabilization}\label{yu-per-1}

This subsection studies the periodic complete stabilization for  the periodic  system $[A(\cdot),B(\cdot)]$, i.e.,
\begin{equation}\label{contsyst}
y'(t) = A(t)y(t)+B(t)u(t),\;\; t\in \mathbb{R}^+,
\end{equation}
under  the following hypotheses:
\begin{enumerate}
\item[($\widehat{H_1}$)]
For a.e. $t\in \mathbb{R}^+$,
$A(t):=A+D(t)$, where the operator $A$ generates a  $C_0$-semigroup $S(t)$ ($t\geq 0$) on  $X$;  the operator-valued function $D(\cdot)$ belongs to
$L^1_{loc}(\mathbb{R}^+;{\cal L}(X))$ and is $\mathbb{T}$-periodic ($\mathbb{T}>0$), i.e., $D(t+\mathbb{T})=D(t)$ for a.e. $t\in \mathbb{R}^+$.
\item[($\widehat{H_2}$)]  The operator-valued function $B(\cdot)$ belongs to $L^\infty(\mathbb{R}^+;{\cal L}(U;X))$ and is $\mathbb{T}$-periodic, i.e., $B(t+\mathbb{T})=B(t)$ for a.e. $t\in \mathbb{R}^+$.
 \end{enumerate}
By \cite[Chapter 1, Proposition 1.2]{WX2}, we have what follows: First,
   $A(\cdot)$ generates a unique  $\mathbb{T}$-periodic evolution $\Phi(\cdot,\cdot)$
   (i.e., $\Phi(t+\mathbb{T}, s+\mathbb{T})=\Phi(t,s)$ for all $0\leq s\leq t$) on $X$;
   Second, for each $\mathbb{T}$-periodic feedback operator $K\in L^\infty(\mathbb{R}^+;{\cal L}(X;U))$
    (i.e., $K(\mathbb{T}+t)=K(t)$ for a.e. $t\in \mathbb{R}^+$),
 $A_{K}(\cdot):= A(\cdot) +B(\cdot)K(\cdot)$
generates a unique  $\mathbb{T}$-periodic evolution  $\Phi_{K}(\cdot,\cdot)$;
Third, when
 $u\in L^2(\mathbb{R}^+;U)$ and $y_0\in X$,
  \begin{equation*}
 y(t;u,y_0)=\Phi(t,0)y_0+\int_0^t\Phi(t,s)Bu(s)ds\;\;\mbox{and}\;\; y_K(t;y_0)=\Phi_{K}(t,0)y_0\;\mbox{for all}\;\;t\geq 0,
 \end{equation*}
 where $y(\cdot;u,y_0)$ is the solution
 to (\ref{contsyst}) with the initial condition: $y(0)=y_0$, while
 $y_K(\cdot;z)$ is the solution to the equation:
$y'(t)=[A(t)+B(t)K(t)]y(t),\; t\geq 0;\;\; y(0)=y_0$.

 To present of the main result of this subsection, we need the following definitions:
\begin{definition}\label{definition 3.7,11-24}
\begin{enumerate}
\item[(i)] The system (\ref{contsyst}) is said to be periodically  stabilizable,  if  there exists $C>0$,
$\alpha>0$ and a $\mathbb{T}$-periodic feedback operator
$K(\cdot)\in L^\infty(\mathbb{R}^+;{\cal L}(X;U))$  such that
$\|\Phi_{K}(t,0)\|_{\mathcal{L}(X)}\leq C e^{-\alpha t}$ for all $t\geq 0$.

\item[(ii)] The system (\ref{contsyst}) is said to be periodically completely stabilizable,  if for any $\alpha\in\mathbb{R}^+$, there exists $C:=C(\alpha)>0$ and a $\mathbb{T}$-periodic feedback operator
$K(\cdot):=K_\alpha(\cdot)\in L^\infty(\mathbb{R}^+;{\cal L}(X;U))$ such that
$\|\Phi_{K}(t, 0)\|_{\mathcal{L}(X)}\leq C e^{-\alpha t}$ for all $t\geq 0$.
\end{enumerate}
\end{definition}
The main result of this subsection is as follows:
\begin{theorem}\label{maintheorem1}
Suppose that $(\widehat{H_1})$ and $(\widehat{H_2})$ are true. Then the following statements are equivalent:
\begin{enumerate}
\item[$(i)$] The  system (\ref{contsyst}) is periodically completely stabilizable.
\item[$(ii)$] For any $k\in\mathbb{N}$, there exists  $n_k\in\mathbb{N}$ and $ C(k)>0$  such that
\begin{equation*}
\|\Phi(n_k\mathbb{T},0)^*\psi\|_X\leq
	C(k)\|B(\cdot)^*\Phi(n_k \mathbb{T},\cdot)^*\psi\|_{L^2(0,n_k\mathbb{T};U)}+e^{-kn_k\mathbb{T}}\|\psi\|_X\;\;\mbox{for any}\;\;\psi\in X.
\end{equation*}
\end{enumerate}
\end{theorem}

To prove Theorem \ref{maintheorem1}, we need the next Lemma \ref{10-19.lemma1}, which  is quoted from
 \cite[ Theorem 1.1]{Xu}.

\begin{lemma}\label{10-19.lemma1}
Suppose that $(\widehat{H_1})$ and $(\widehat{H_2})$ hold. Let $\mu\geq 0$. Let $\Phi^\mu(\cdot,\cdot)$ be the $\mathbb{T}$-periodic evolution generated by
	$A(\cdot) +\mu I$. Then the following statements are equivalent:
\begin{enumerate}
\item[$(i)$] The following system  is periodically stabilizable:
\begin{equation}\label{Liu-12-7-01}
y'(t) = (A(t)+\mu I)y(t)+B(t)u(t),\;\; t\in \mathbb{R}^+.
\end{equation}
\item[$(ii)$]  There exists $ \delta\in(0,1)$,  $n\in\mathbb{N}$  and $C(n)>0$ such that
\begin{equation}\label{Liu-12-7-02}
\|\Phi^\mu(n\mathbb{T},0)^*\psi\|_X\leq
	C(n)\|B(\cdot)^*\Phi^\mu(n \mathbb{T},\cdot)^*\psi\|_{L^2(0,n\mathbb{T};U)}+\delta\|\psi\|_X\;\;\mbox{for any}\;\;\psi\in X.
\end{equation}
\end{enumerate}
\end{lemma}
\begin{remark}\label{remark3.11,12.8}
 Given $n\in \mathbb{N}$ and $\psi\in X$, we have that
$\Phi^\mu(n\mathbb{T}, t)^*\psi=\varphi_n(t;\psi)$
for each $ t\in[0, n\mathbb{T}]$, where $\varphi_n(\cdot;\psi)$ is the solution to the  equation:
$\varphi'_n(t)=-(A+\mu I)(t)^*\varphi_n(t)\; t\in[0, n\mathbb{T}]$; $\;\varphi_n(n\mathbb{T})=\psi$.
In \cite[ Theorem 1.1]{Xu}, \eqref{Liu-12-7-02} is expressed in terms of $\varphi_n(\cdot;\psi)$.
\end{remark}
Now, we are in the position to prove Theorem \ref{maintheorem1}.

\begin{proof}[The proof of Theorem \ref{maintheorem1}]
First of all,  when $\mu\geq 0$ and
$K(\cdot)\in L^\infty(\mathbb{R}^+;{\cal L}(X,U))$ is $\mathbb{T}$-periodic, we have
\begin{equation}\label{12.6-1}
	   \Phi^\mu(t,s)=e^{\mu (t-s)} \Phi(t,s)\;\;\mbox{and}\;\;
\Phi_K^\mu(t,s)=e^{\mu (t-s)}\Phi_K(t,s)\;\;\mbox{for all}\;\; 0\leq s\leq t,
	\end{equation}
where $\Phi_K^\mu(\cdot,\cdot)$ is the $\mathbb{T}$-periodic evolution generated by $A(\cdot)+\mu I+B(\cdot)K(\cdot)$.
We now  organize the rest of the  proof in  two steps.
\vskip 5pt

\noindent {\it Step 1. We show $(i)\Rightarrow (ii)$.}
\vskip 5pt
Suppose $(i)$ is true.
 Arbitrarily fix $k\in\mathbf{N}$.
Then according to Definition \ref{definition 3.7,11-24},
  there exists $C_{k}>0$ and
 a $\mathbb{T}$-periodic feedback operator  $K(\cdot):=K_{k}(\cdot)\in L^\infty(\mathbb{R}^+;{\cal L}(X;U))$  such that
$\|\Phi_{K}(t, 0)\|_{\mathcal{L}(X)}\leq C_{k} e^{-(k+1)  t}$ for all $t\geq 0$,
which, along with the second equality in \eqref{12.6-1},   implies that
$\|\Phi_{K}^k(t, 0)\|_{\mathcal{L}(X)}\leq C_k e^{-t}$ for all $t\geq 0$.
This, along with  Definition \ref{definition 3.7,11-24}, leads to the periodic stabilizability of
     the system \eqref{Liu-12-7-01} (where $\mu=k$).
Then according to  Lemma \ref{10-19.lemma1},
there exists $ \delta_k\in(0,1)$,  $n_k\in\mathbb{N}$  and $C(k)>0$ such that
\begin{equation*}
	\|\Phi^k(n_k\mathbb{T},0)^*\psi\|_X\leq
	C(k)\|B(\cdot)^*\Phi^k(n_k \mathbb{T}, \cdot)^*\psi\|_{L^2(0,n_k\mathbb{T};U)}+\delta_k\|\psi\|_X\;\;\mbox{for any}\;\;\psi\in X.
\end{equation*}
	This, together with the first equality in \eqref{12.6-1}, implies that
\begin{eqnarray*}
	\|\Phi(n_k\mathbb{T},0)^*\psi\|_X&\leq&
	C(k)\|e^{-k\cdot}B(\cdot)^*\Phi(n_k \mathbb{T}, \cdot)^*\psi\|_{L^2(0,n_k\mathbb{T};U)}+\delta_ke^{-kn_k\mathbb{T}}\|\psi\|_X\\
	&\leq& C(k)\|B(\cdot)^*\Phi(n_k \mathbb{T}, \cdot)^*\psi\|_{L^2(0,n_k\mathbb{T};U)}+e^{-kn_k\mathbb{T}}\|\psi\|_X
	\;\;\mbox{for any}\;\;\psi\in X,
\end{eqnarray*}	
	which leads to $(ii)$.
	\vskip 5pt
	
	\noindent {\it Step 2. We show $(ii)\Rightarrow (i)$.}
	\vskip 5pt
	Suppose that $(ii)$ is true. Arbitrarily fix $\mu>0$. We first show that the system
\eqref{Liu-12-7-01} is periodically stabilizable. To this end, we  take $k=[\mu]+1$, where $[\mu]$ denotes the integer part of $\mu$.
 Then by  $(ii)$, we can find
	  $n_k\in\mathbb{N}$ and $ C(k)>0$  such that
	\begin{equation*}
			\|\Phi(n_k\mathbb{T},0)^*\psi\|_X\leq
 C(k)\|B(\cdot)^*\Phi(n_k \mathbb{T},\cdot)^*\psi\|_{L^2(0,n_k\mathbb{T};U)}+e^{-kn_k\mathbb{T}}\|\psi\|_X
		\;\;\mbox{for any}\;\;\psi\in X.
	\end{equation*}
		Hence,  by the first equality in \eqref{12.6-1} and the same way to prove (\ref{10-26Liu001}), we have
	\begin{eqnarray*}
		\|\Phi^\mu(n_k\mathbb{T},0)^*\psi\|_X&\leq&
		C(k)\|e^{\mu \cdot}B(\cdot)^*\Phi^\mu(n_k \mathbb{T};\cdot)^*\psi\|_{L^2(0,n_k\mathbb{T};U)}+e^{-(k-\mu)n_k\mathbb{T}}\|\psi\|_X\\
		&\leq& C(k)e^{\mu n_k\mathbb{T}}
		\|B(\cdot)^*\Phi^\mu(n_k \mathbb{T}, \cdot)^*\psi\|_{L^2(0,n_k\mathbb{T};U)}+e^{-(k-\mu)n_k\mathbb{T}}\|\psi\|_X
		\;\;\mbox{for any}\;\;\psi\in X.
	\end{eqnarray*}	
	Since $e^{-(k-\mu)n_k\mathbb{T}}<1$, the above, along with Lemma \ref{10-19.lemma1}, yields
that the system \eqref{Liu-12-7-01} is periodically stabilizable.

We next show that  the system \eqref{Liu-12-7-01} is periodically completely stabilizable.
Indeed, by the periodic stabilizability of the system \eqref{Liu-12-7-01} and by \eqref{12.6-1}, one can easily check
that there is $C=C(\mu)>0$ and  $\mathbb{T}$-periodic   $K(\cdot)\in L^\infty(\mathbb{R}^+;{\cal L}(X;U))$ (which depends on $\mu$)  so that $\|\Phi_K(t,0)\|_{\mathcal{L}(X)}\leq Ce^{-\mu t}$ for all $t\geq 0$.
Then, since $\mu>0$ was arbitrarily taken, we get, from Definition \ref{definition 3.7,11-24},
the periodic complete stabilizability of the system \eqref{Liu-12-7-01}.
\end{proof}

\section{Applications}\label{yu-section-app-4}

In this section, we present  several examples of control systems, which are not null controllable, but
can be shown to be  completely stabilizable, through verifying the weak observability inequalities in
 Theorem \ref{yu-theorem-3-4-1}, as well as its extensions presented in Section \ref{yu-sect-3}.
 Besides, we  give some sufficient conditions ensuring the weak observability inequalities, from the
 perspective of the spectral projection. These conditions
  not only are useful in the studies of our examples, but also have independent interest.

\subsection{Conditions ensuring the weak observability}\label{yu-suff-1}
This subsection presents two theorems. One is about the setting in Section \ref{yu-section-1}, while another is about the setting in Subsection \ref{yu-sec-10-18-1}.
\begin{theorem}\label{yu-theorem-10-9-1}
    Let $A$, with its domain $D(A)$, generate a $C_0$-semigroup $S(t)$ $(t\geq0)$ on $X$ and
    $B\in\mathcal{L}(U;X)$.  Let $\{P_k\}_{k\in\mathbb{N}}$ be a family of orthogonal projections on $X$. Suppose  the following dissipative inequality is satisfied:
    \begin{enumerate}
  \item [(a)] For each $k\in\mathbb{N}$, there exists $M_k>0$ and $\alpha_k>0$, with $\alpha_k\rightarrow +\infty$ as $k\rightarrow +\infty$, such that
\begin{equation}\label{yu-10-9-10}
    \|(I-P_k)S(t)^*\varphi\|_X\leq M_ke^{-\alpha_k t}\|\varphi\|_X\;\;\mbox{for any}\;\;t\in\mathbb{R}^+\mbox{and}\;\;\varphi\in X.
\end{equation}
\end{enumerate}
    Then the statement (iii) in Theorem \ref{yu-theorem-3-4-1} holds when one of the following two conditions $(b_1)$ and $(b_2)$ is true:
\begin{enumerate}
  \item [($b_1$)]  For each $k\in \mathbb{N}$, there exists $C_k>0$ such that the following  spectral inequality is true:
\begin{equation}\label{yu-10-9-3}
    \|P_k\varphi\|_X\leq C_k\|B^*P_k\varphi\|_U\;\;\mbox{for any}\;\;\varphi\in X;
\end{equation}
  \item [($b_2$)] There exists  $T_0> 0$ such that for each $k\in\mathbb{N}$, there exists $C(k,T_0)>0$ such that the following truncated observability inequality holds:
\begin{equation}\label{yu-10-10-40}
    \|P_kS(T_0)^*\varphi \|_X^2\leq C(k,T_0)\int_0^{T_0}\|B^*P_kS(t)^*\varphi\|_U^2dt\;\;\mbox{for any}\;\;\varphi\in X.
\end{equation}
\end{enumerate}
\end{theorem}
\begin{proof} Since $S(t)$ $(t\geq0)$ is a  $C_0$-semigroup,  there exists  $M> 1$ and $\delta_0> 0$ such that
    \begin{equation}\label{10.26.01}
    \|S(t)\|_{\mathcal{L}(X)}\leq M e^{\delta_0 t}\;\;\mbox{for all} \;\;t\in\mathbb{R}^+.
\end{equation}
Arbitrarily fix  $\alpha>0$. Then fix $k\in\mathbb{N}$ such that $\alpha_k>\alpha$. (Such $k$ exists, since $\alpha_k$ tends to $+\infty$.) We shall prove statement $(iii)$ in Theorem \ref{yu-theorem-3-4-1}
 by two cases.

 We first consider  the case when  $(b_1)$ is true. Arbitrarily fix  $T>1$ and $\varphi\in X$.
    By \eqref{10.26.01} and  the Cauchy-Schwarz inequality, we find
  \begin{eqnarray}\label{10.26.02}
   \|S(T)^*\varphi\|^2_X&=&\left\|\int^1_0S(t)^*S(T-t)^*\varphi\mathrm dt\right\|^2_X\leq\left(\int^1_0\|S(t)^*\|_{\mathcal{L}(X)}\|S(T-t)^*\varphi\|_X\mathrm dt\right)^2\nonumber\\
    &\leq& M^2e^{2\delta_0 } \left(\int^1_0\|S(T-t)^*\varphi\|_X\mathrm dt\right)^2\leq M^2e^{2\delta_0 } \int^1_0\|S(T-t)^*\varphi\|_X^2\mathrm dt.
 \end{eqnarray}
Meanwhile, since $P_k$ is an orthogonal projection, it follows from
  (\ref{yu-10-9-3}) and (\ref{yu-10-9-10})
that for each $t\in[0,T]$,
    \begin{eqnarray}\label{10.26.05}
      &\;&\|S(T-t)^*\varphi\|_X^2=\|P_kS(T-t)^*\varphi\|_X^2
      +\|(I-P_k)S(T-t)^*\varphi\|_X^2\nonumber\\
     &\leq&C_k^2\|B^*P_kS(T-t)^*\varphi\|_U^2+\|(I-P_k)S(T-t)^*\varphi\|_X^2\nonumber\\
      &\leq&2C_k^2\|B^*S(T-t)^*\varphi\|_U^2
      +2C_k^2\|B^*(I-P_k)S(T-t)^*\varphi\|_U^2+\|(I-P_k)S(T-t)^*\varphi\|_X^2\nonumber\\
    &\leq&2C_k^2\|B^*S(T-t)^*\varphi\|_U^2+\left[2C_k^2\|B\|_{\mathcal{L}(U; X)}^2+1\right]\|(I-P_k)S(T-t)^*\varphi\|_X^2\nonumber\\
  &\leq&2C_k^2\|B^*S(T-t)^*\varphi\|_U^2+\left[2C_k^2\|B\|_{\mathcal{L}(U; X)}^2+1\right]M_k^2e^{-2\alpha_k(T-t)}\|\varphi\|_X^2.
     \end{eqnarray}
    Next, integrating both sides of \eqref{10.26.05} for $t$ over
     $(0,1)$, using (\ref{10.26.02}), and noting  that $\alpha_k>\alpha$, we obtain
      \begin{equation*}
 \|S(T)^*\varphi\|_X^2
 \leq 2 M^2C_k^2e^{2\delta_0 } \int^1_0\|B^*S(T-t)^*\varphi\|_U^2\mathrm dt
 + M^2M_k^2e^{2(\delta_0+\alpha)} \left[2C_k^2\|B\|_{\mathcal{L}(U; X)}^2+1\right]e^{-2\alpha T}\|\varphi\|_X^2,
 \end{equation*}
 which, along with the fact that $T>1$, leads to
 \begin{equation}\label{10.26.08}
 \|S(T)^*\varphi\|_X
 \leq D(\alpha) \|B^*S(T-\cdot)^*\varphi\|_{L^2(0, T; U)}
 + C(\alpha)e^{-\alpha T}\|\varphi\|_X, %
 \end{equation}
 where
 \begin{equation*}
 D(\alpha):=\sqrt{2}MC_ke^{\delta_0 },\;\;C(\alpha):= MM_ke^{\delta_0+\alpha}\sqrt{ \left[2C_k^2\|B\|_{\mathcal{L}(U; X)}^2+1\right]}.
 \end{equation*}
 Since $k$ depends only on $\alpha$, the statement  $(iii)$ in Theorem \ref{yu-theorem-3-4-1}
 follows from \eqref{10.26.08} in the current case.

 We next consider the case where  $(b_2)$ holds. Arbitrarily fix  $T\geq 2T_0$. Then, there exists
 a natural number  $N\geq 2$ such that $NT_0\leq T<(N+1)T_0$.
 By \eqref{10.26.01}, \eqref{yu-10-10-40}  and \eqref{yu-10-9-10}, we see that
 for each  $\varphi\in X$,
\begin{eqnarray}\label{yu-10-10-44}
   \|S(T)^*\varphi\|_X^2&=& \|S(T-NT_0)^*S(NT_0)^*\varphi\|_X^2\leq M^2e^{2\delta_0 T_0}\|S(NT_0)^*\varphi\|_X^2\nonumber\\
    &=& M^2e^{2\delta_0 T_0}(\|P_kS(T_0)^*S((N-1)T_0)^*\varphi\|_X^2+\|(I-P_k)S(NT_0)^*\varphi\|_X^2)\nonumber\\
    &\leq& M^2e^{2\delta_0 T_0}\left(C(k, T_0)\int_0^{T_0}\left\|B^*P_kS
    ((N-1)T_0+t)^*\varphi\right\|^2_Udt
    +\|(I-P_k)S(NT_0)^*\varphi\|_X^2\right)\nonumber\\
    &\leq& 2M^2e^{2\delta_0 T_0}C(k, T_0)\left(\int_{(N-1)T_0}^{NT_0}\left\|B^*S
    (t)^*\varphi\right\|^2_Udt+\int_{(N-1)T_0}^{NT_0}\left\|B^*(I-P_k)S
    (t)^*\varphi\right\|^2_Udt\right)\nonumber\\
    &&+M^2e^{2\delta_0 T_0}M_k^2e^{-2\alpha_kNT_0}\|\varphi\|_X^2.
\end{eqnarray}
Meanwhile, it follows  from \eqref{yu-10-9-10} that
\begin{eqnarray}\label{yu-10-18-6-1-1}
    \int_{(N-1)T_0}^{NT_0}\left\|B^*(I-P_k)S(t)^*\varphi\right\|^2_Udt
    &\leq& \|B\|^2_{\mathcal{L}(U; X)} M_k^2  \int_{(N-1)T_0}^{NT_0}e^{-2\alpha_k t}\|\varphi\|_X^2dt\nonumber\\
    &\leq& \|B\|^2_{\mathcal{L}(U; X)}M_k^2 T_0  e^{-2\alpha_k (N-1)T_0}\|\varphi\|_X^2.
\end{eqnarray}
Using \eqref{yu-10-10-44} and \eqref{yu-10-18-6-1-1}, noting that  $\alpha_k>\alpha$ and $NT_0\leq T<(N+1)T_0$,  we get
\begin{eqnarray}\label{yu-10-18-6-1}
    \|S(T)^*\varphi\|_X^2\leq D(\alpha)^2\int_{0}^{T}\left\|B^*S(t)^*\varphi\right\|^2_Udt+ C(\alpha)^2e^{-2\alpha T}\|\varphi\|_X^2,
\end{eqnarray}
where
\begin{eqnarray*}
D(\alpha):= Me^{\delta_0 T_0}\sqrt{2C(k, T_0)},\;\; C(\alpha):=M M_k e^{(\delta_0+\alpha)T_0}\sqrt{2C(k, T_0)\|B^*\|^2_{\mathcal{L}(X; U)}T_0e^{2\alpha T_0}+1}.
\end{eqnarray*}
Now, the statement  $(iii)$ in Theorem \ref{yu-theorem-3-4-1}
 follows from \eqref{yu-10-18-6-1} in the current case.
\end{proof}

\begin{remark}
     (i) The  Lebeau-Robbiano strategy says in plain language that
     the null controllability can be implied by
     a spectral inequality and a dissipative inequality. In this strategy,
     the following compatibility condition on the decay rate in the dissipative inequality
     and the growth rate in the spectral inequality is necessary: the former is greater than the latter
      (see e.g. \cite[Theorem 2.2]{Miller1}).  In the studies of the exponential stabilizability, such compatibility condition is relaxed
     (see \cite[Lemma 2.2]{Huang-Wang-Wang}). Our Theorem~\ref{yu-theorem-10-9-1}, together with   Theorem~\ref{yu-theorem-3-4-1}, improves \cite[Lemma 2.2]{Huang-Wang-Wang} from two perspectives: First, it serves for the complete stabilizability; Second, there is no any compatibility condition on the constants in (\ref{yu-10-9-10}) and~(\ref{yu-10-9-3}).

      (ii) When $P_kS(\cdot)=S(\cdot)P_k$,  the condition  $(b_1)$ implies condition  $(b_2)$.
      This can be checked directly.

(iii) We borrowed the name ``truncated observability inequality''  from \cite{Barbu}, where such kind of observability inequality is applied to construct an internal feedback control
 stabilizing  the Navier-Stokes equations.
\end{remark}

To show the similar result to Theorem \ref{yu-theorem-10-9-1} in the setting where $B$ is unbounded, we need
the next lemma.

\begin{lemma}\label{Liu-11-18-remark}
Suppose that  $[A,B]$ satisfies the following
conditions:
\begin{enumerate}
  \item[(a)] The operator $A$, with its domain $D(A)$,  generates an analytic semigroup $S(t)\ (t\geq 0)$ on $X$.
   \item [(b)]
   There exists $\gamma\in \left(0,\frac{1}{2}\right)$
    such that  $B\in\mathcal{L}(U;X_{-\gamma})$,
    where
    $X_{-\gamma}$  is the completion of $X$ with respect to the norm $\|z\|_{-\gamma}:=\|(\rho_0I-A)^{-\gamma}z\|_X$, $z\in X$ (where $\rho_0\in \rho(A)\cap \mathbb{R}$ is arbitrarily fixed).
\end{enumerate}
    Then the following conclusions are true:
\begin{enumerate}
  \item [(i)] The assumptions $(\widetilde{H_1})$-$(\widetilde{H_3})$ in Section \ref{yu-sec-10-18-1}
  hold for $[A,B]$.
  \item [(ii)] The operator $(\rho_0I-A)^\gamma$ has a unique extension
     $\widetilde{(\rho_0I-A)^\gamma}\in
   \mathcal{L}(X;X_{-\gamma})$.
   Moreover, this extension is invertible and  $(\widetilde{(\rho_0I-A)^\gamma})^{-1}B
   \in\mathcal{L}(U;X)$.
\end{enumerate}
\end{lemma}
\begin{proof}
  We organize the proof by two steps.

  \vskip 5pt

  \noindent{\it Step 1. We show the conclusion $(i)$. }
  \vskip 5pt

   First,  $(\widetilde{H_1})$ is clearly true.
\par

   Second, one can directly check that
   $X_{-\gamma}$ is continuously embedded into $X_{-1}$. Then by the condition   $(b)$, we get $B\in \mathcal{L}(U; X_{-1})$ which leads to $(\widetilde{H_2})$.

\par

 We are now going to show $(\widetilde{H_3})$.
 First of all,  one can directly check
   the following two facts:
\begin{enumerate}
\item[($f_1$)] The operator $(\rho_0I-A)^\gamma$ belongs to $\mathcal{L}(D((\rho_0I-A)^\gamma);X)$ (Here, the norm of $D((\rho_0I-A)^\gamma)$ is as:  $\|z\|_{D((\rho_0I-A)^\gamma)}=\|(\rho_0I-A)^\gamma z\|_X$, $z\in D((\rho_0I-A)^\gamma)$.)  and  has a unique extension
    $\widetilde{(\rho_0I-A)^\gamma}\in \mathcal{L}(X;X_{-\gamma})$ which is invertible (see \cite[Chapter 2, Proposition 2.10.3]{Tucsnak-Weiss});
   \item [($f_2$)] $B\in \mathcal{L}(U;X_{-\gamma})$ if and only if $\Big(\widetilde{(\rho_0I-A)^\gamma}\Big)^{-1}B\in \mathcal{L}(U;X)$.
   \end{enumerate}
    The above facts $(f_1)$ and $(f_2)$, together with $D(A^*)\subset D((\rho_0I-A^*)^\gamma)$ and the analyticity of $S(t)$ ($t\geq 0$) (which means that $S(t)^*x\in D(A^*)$ for any $x\in X$ when $t>0$), yield that for each $t>0$,
\begin{eqnarray}\label{yu-5-8-1}
    B^*S(t)^*x&=&B^*\left(\widetilde{(\rho_0I-A)^\gamma}^*\right)^{-1}
    \widetilde{(\rho_0I-A)^\gamma}^*S(t)^*x\nonumber\\
    &=&B^*\left(\widetilde{(\rho_0I-A)^\gamma}^*\right)^{-1}(\rho_0I-A^*)^\gamma S(t)^*x
    \;\;\mbox{for any}\;\;x\in X.
\end{eqnarray}
    Here, we notice that $B^*\in \mathcal{L}(D(A^*);U)$ since $B\in \mathcal{L}(U;X_{-1})$ is proved and $X_{-1}$ is the dual space of $D(A^*)$ with respect to the pivot space $X$ (see $(d_3)$ in Remark \ref{yu-remark-12-1}).
\par
 Meanwhile,
 we have the following observations: First,
 the analytic semigroup $S(t)^*e^{-\rho_0 t}$ $(t\geq0)$) is generated by
   $-\rho_0I+A^*$; Second,
  since $\rho_0\in \rho(A)\cap\mathbb{R}$, we have   $0\in \rho(\rho_0I-A^*)$ (see \cite[Chapter 1, Lemma 10.2]{Pazy}).
    From these observations, we can use \cite[Chapter 2, Theorem 6.13]{Pazy} to find
    $C(\gamma)>0$ such that
\begin{equation}\label{Liu-11-15-01}
\|(\rho_0  I-A^*)^{\gamma}S(t)^*\|_{\mathcal{L}(X)}\leq C(\gamma)e^{\rho_0 t}t^{-\gamma}\;\;\mbox{for all}\;\; t>0.
\end{equation}

Now, it follows by \eqref{Liu-11-15-01} and (\ref{yu-5-8-1})  that, when $T>0$ and $x\in X$,
\begin{eqnarray*}
\int_0^T\|B^* S(t)^*x\|^2_Udt
&=&\int_0^T\Big\|B^*\Big(\widetilde{(\rho_0I-A)^\gamma}^*\Big)^{-1}(\rho_0 I-A^*)^\gamma S(t)^*x\Big\|^2_Udt\nonumber\\
&\leq& \Big\|B^*\Big(\widetilde{(\rho_0I-A)^\gamma}^*\Big)^{-1}
\Big\|_{\mathcal{L}(X;U)}^2C(\gamma)^2e^{2\rho_0 T}\int_0^T\frac{1}{t^{2\gamma}}dt\|x\|^2_X\leq C(T, \gamma)\|x\|^2_X,
\end{eqnarray*}
where $C(T,\gamma):=\Big\|B^*\Big(\widetilde{(\rho_0I-A)^\gamma}^*
\Big)^{-1}\Big\|^2_{\mathcal{L}(X;U)}C(\gamma)^2e^{2\rho_0 T}\frac{T^{1-2\gamma}}{1-2\gamma}$. This leads to $(\widetilde{H_3})$ obviously.
\par

\vskip 5pt

  \noindent{\it Step 2. The conclusion $(ii)$ follows from the above $(f_1)$ and $(f_2)$ at once. }
\end{proof}
\begin{theorem}\label{yu-corollary-10-10-1}

Assume that the conditions $(a)$ and $(b)$ in Lemma \ref{Liu-11-18-remark} hold.
Suppose that there is a family of
orthogonal projections $\{P_k\}_{k\in\mathbb{N}}$ on $X$  satisfying the commutative condition:
 \begin{equation}\label{yu-10-10-41}
P_kS(\cdot)=S(\cdot)P_k;
\end{equation}
 the dissipative condition  $(a)$ in Theorem \ref{yu-theorem-10-9-1};
  and the following observability condition:
  \begin{enumerate}
    \item [$(b'_2)$] there exists $T_0> 0$ such that for each $k\in\mathbb{N}$, there is $C(k,T_0)>0$ such that
\begin{equation}\label{yu-12-2-1}
    \|P_kS(T_0)^*\varphi \|_X^2\leq C(k,T_0)\int_0^{T_0}\|B^*P_kS(t)^*\varphi\|_U^2dt\;\;\mbox{for any}\;\;\varphi\in D(A^*).
\end{equation}
  \end{enumerate}
 Then the statement (iii) in Theorem \ref{yu-theorem-10-5-1} is true.
\end{theorem}

\begin{proof}
Let $T_0$ and $C(k,T_0)$ be given in $(b_2')$. Let $M_k$ and $\alpha_k$ be
given in $(a)$ of Theorem \ref{yu-theorem-10-9-1}.
Let $M$ and $\delta_0$ be given by (\ref{10.26.01}) which clearly holds in the current case.
Arbitrarily fix
 $\alpha>0$ and $T\geq 2T_0$. Then, there
 is a natural number $N$ with  $N\geq 2$ such that $NT_0\leq T<(N+1)T_0$.

  By Lemma \ref{Liu-11-18-remark}, (\ref{yu-10-10-41}) and (\ref{yu-12-2-1}), and by the same way as that used in the proof of  \eqref{yu-10-10-44} (in the proof of Theorem \ref{yu-theorem-10-9-1}), we can
  easily get
\begin{eqnarray}\label{Liu-11-16-01}
   \|S(T)^*\varphi\|_X^2&\leq&2M^2e^{2\delta_0 T_0}C(k, T_0)\left(\int_{(N-1)T_0}^{NT_0}\left\|B^*S
    (t)^*\varphi\right\|^2_Udt+\int_{(N-1)T_0}^{NT_0}\left\|B^*S
    (t)^*(I-P_k)\varphi\right\|^2_Udt\right)\nonumber\\
    &&+M^2e^{2\delta_0 T_0}M_k^2e^{-2\alpha_kNT_0}\|\varphi\|_X^2\;\;\mbox{for any}\;\; \varphi\in D(A^*).
\end{eqnarray}
    (Notice that
    it follows by  \eqref{yu-10-10-41} and  the property of $C_0$-semigroup that $P_kA^*=A^*P_k$ in $D(A^*)$ for each $k\in\mathbb{N}$, which implies that if $\varphi\in D(A^*)$ then $P_k\varphi\in D(A^*)$.) At the same time, we clearly have (\ref{yu-5-8-1}) and \eqref{Liu-11-15-01}. These, along with
               \eqref{yu-10-9-10} and the conclusion $(ii)$ in Lemma \ref{Liu-11-18-remark}, yield
\begin{eqnarray}\label{yu-10-18-6w}
    &&\int_{(N-1)T_0}^{NT_0}\left\|B^*S(t)^*(I-P_k)\varphi\right\|^2_Udt\nonumber\\
    &=&\int_{(N-1)T_0}^{NT_0}
    \|((\widetilde{(\rho_0I-A)^{\gamma}})^{-1} B)^*(\rho_0I-A^*)^\gamma S(T_0)^*S(t-T_0)^*(I-P_k)\varphi\|_U^2dt\nonumber\\
    &\leq&\|(\widetilde{(\rho_0I-A)^{\gamma}})^{-1} B\|^2_{\mathcal{L}(U;X)}  C(\gamma)^2e^{2\rho_0 T_0}T_0^{-2\gamma}\int_{(N-2)T_0}^{(N-1)T_0}\left\|(I-P_k)S(t)^*\varphi\right\|^2_Xdt\nonumber\\
    &\leq&\|(\widetilde{(\rho_0I-A)^{\gamma}})^{-1} B\|^2_{\mathcal{L}(U;X)}  C(\gamma)^2e^{2\rho_0 T_0}T_0^{-2\gamma} M_k^2\int_{(N-2)T_0}^{(N-1)T_0} e^{-2\alpha_k t}dt\|\varphi\|^2_X\nonumber\\
    &\leq& \|(\widetilde{(\rho_0I-A)^{\gamma}})^{-1} B\|^2_{\mathcal{L}(U;X)}  C(\gamma)^2e^{2\rho_0 T_0}T_0^{1-2\gamma}M_k^2 e^{-2\alpha_k (N-2)T_0}\|\varphi\|^2_X
    \;\;\mbox{for any}\;\;\varphi\in D(A^*).
\end{eqnarray}
Now, by  \eqref{Liu-11-16-01} and \eqref{yu-10-18-6w}, using  the facts $\alpha_k>\alpha$ and $NT_0\leq T<(N+1)T_0$,  we obtain
\begin{eqnarray*}
    \|S(T)^*\varphi\|_X^2\leq D(\alpha)^2\int_{0}^{T}\left\|B^*S(t)^*\varphi\right\|^2_Udt+ C(\alpha)^2e^{-2\alpha T}\|\varphi\|_X^2\;\;\mbox{for any}\;\;\varphi\in D(A^*),
\end{eqnarray*}
where $D(\alpha):= Me^{\delta_0 T_0}\sqrt{2C(k, T_0)}$ and
\begin{eqnarray*}
C(\alpha):=M M_k e^{(\delta_0+\alpha)T_0}\sqrt{2C(k, T_0)\|(\widetilde{(\rho_0I-A)^{\gamma}})^{-1}B\|^2_{\mathcal{L}(U;X)}C(\gamma)^2
e^{2(\rho_0+2\alpha) T_0}T_0^{1-2\gamma}+1}.
\end{eqnarray*}
 This leads to  the statement $(iii)$ in Theorem \ref{yu-theorem-10-5-1}.
\end{proof}

\subsection{Examples}\label{yu-exmp-4}
This subsection presents  several concrete controlled equations  which are not null controllable but
 can be checked to be  completely stabilizable,  via the weak observability inequalities.
The first two examples are under the framework in Section \ref{yu-section-1}, the third example is under the framework
in Subsection \ref{yu-sec-10-18-1}, while the last example is under the framework in Subsection \ref{yu-per-1}.

\vskip 5pt

\noindent\emph{\textbf{Example 1.  (Fractional heat equations in $\mathbb{R}^n$.)}}\ \
Let $c\geq 0$ and $s\in (0,1)$.
Let $E$ be a thick set in $\mathbb{R}^n$ such that $E\subset B^c(\bar{x},R)$ for some $\bar{x}\in \mathbb{R}^n$ and $R>0$ (where $B^c(\bar{x},R)$ denotes the complementary set of the closed ball centered at $\bar{x}\in \mathbb{R}^n$ and of radius $R$).
Here, by $E$ being a thick set in $\mathbb{R}^n$, we mean that
 there is $\gamma>0$ and $L>0$ such that
$$|E\cap Q_L(x)|\geq \gamma L^n \ \mathrm{for\ each}\  x\in \mathbb{R}^n,$$
   where $Q_L(x)$ is the closed cube in $\mathbb{R}^n$ (centered at $x$ and of side-length $L$) and $|E\cap Q_L(x)|$ denotes the Lebesgue measure of $E\cap Q_L(x)$.
   We consider the controlled equation:
\begin{equation}\label{p01}
\begin{cases}
\partial_t y(t,x)+(-\triangle)^{\frac{s}{2}}y(t,x)-c y(t,x)=\chi_E(x)u(t,x), &\mbox{in}\;\;\mathbb{R}^+\times \mathbb{R}^n,\\
y(0, \cdot)\in L^2(\mathbb{R}^n).
\end{cases}
\end{equation}
Here,  $\chi_E$  is the characteristic function of $E$,
$u\in L^2(\mathbb{R}^+;L^2(\mathbb{R}^n))$ and the
fractional Laplacian $(-\triangle)^{\frac{s}{2}}$ is defined by
$$
(-\triangle)^{\frac{s}{2}}\varphi:=\mathcal{F}^{-1}(|\xi|^s\mathcal{F}(\varphi)),\;\; \varphi\in C_0^\infty(\mathbb{R}^n),
$$
where and throughout this example, we use $\mathcal{F}$ and $\mathcal{F}^{-1}$ to denote the Fourier transform and its inverse respectively.

For the system \eqref{p01}, we have the following conclusions:

\begin{itemize}
\item The equation \eqref{p01} can be put into our framework (in Section \ref{yu-section-1}) in the following manner:
Let $X=U:=L^2(\mathbb{R}^n)$ and  $A:=-(-\triangle)^{\frac{s}{2}}+c$, with
$D(A)=H^s(\mathbb{R}^n)$. Let $B: L^2(\mathbb{R}^n)\rightarrow L^2(\mathbb{R}^n)$ be defined by $Bv:=\chi_E v$ for each $v\in L^2(\mathbb{R}^n)$. It is well known that $A$ generates a $C_0$-semigroup $\{S(t)\}_{t\geq 0}$ on $X$. One can directly check that $B \in \mathcal{L}(U; X)$.
\item The equation \eqref{p01} with the null control is unstable
(see  \cite[(1.6), as well as the note $(d_1)$]{Huang-Wang-Wang}).

\item Since $E\subset B^c(\bar{x},R)$, it follows by \cite[Theorem 1.3 and its generalization in Section 4.3]{Koenig}
that
the equation \eqref{p01} is not null controllable.

\item  The equation \eqref{p01} is completely stabilizable. (See Theorem \ref{Liuth4-28} below).
\end{itemize}

    The next Theorem \ref{Liuth4-28} may have independent interest.
 \begin{theorem}\label{Liuth4-28}
 Let $\hat E$ be a measurable subset of $\mathbb{R}^n$. Then the equation \eqref{p01}
  (where $E$ is replaced by $\hat E$) is completely stabilizable if and only if $\hat E$ is a thick set.
  \end{theorem}

  \begin{remark}
  Theorem \ref{Liuth4-28} extends \cite[Theorem 1.1]{Huang-Wang-Wang}, which shows that the equation \eqref{p01} (where $E$ is replaced by $\hat E$) is stabilizable if and only if $\hat E$ is a thick set.
  \end{remark}
\begin{proof}[The proof of Theorem \ref{Liuth4-28}]
We first show the {\it only if} part. Suppose that the equation \eqref{p01} (where $E$ is replaced by $\hat E$) is completely stabilizable, then it is  stabilizable. Thus it follows  by  \cite[Theorem 1.1]{Huang-Wang-Wang} that $\hat E$ is a thick set.

 We next show the {\it if} part. Assume that  $\hat E$ is a thick set. We will  use  Theorem \ref{yu-theorem-10-9-1} to show
the statement $(iii)$ of Theorem \ref{yu-theorem-3-4-1} in the following manner:
 Define, for each $k\in \mathbb{N}$, the  linear operator on $L^2(\mathbb{R}^n)$ by
\begin{equation*}
P_k\varphi:=\mathcal{F}^{-1}\left(\chi_{\{\xi\in \mathbb{R}^n: |\xi|^s-c\leq k\}}\mathcal{F}(\varphi)\right),\;\;\varphi\in L^2(\mathbb{R}^n).
\end{equation*}
     By the Plancherel theorem, one can easily check that $\{P_k\}_{k\in\mathbb{N}}$
    are bounded and self-adjoint.

Firstly, we show that $\{P_k\}_{k\in \mathbb{N}}$ are orthogonal projections. To this aim, we first prove $\{P_k\}_{k\in \mathbb{N}}$ are projections. It is suffices to prove that $P_k(I-P_k)=0$ for each $k\in\mathbb{N}$. Indeed, for any
$\varphi, \psi\in L^2(\mathbb{R}^n)$, by the Plancherel theorem and the fact $\{P_k\}_{k\in\mathbb{N}}$ are self-adjoint, we have that, for each $k\in\mathbb{N}$,
\begin{eqnarray*}
    \langle P_k(I-P_k)\varphi,\psi\rangle_{L^2(\mathbb{R}^n)}
    &=&\langle (I-P_k)\varphi,P_k\psi\rangle_{L^2(\mathbb{R}^n)}
    =\langle\mathcal{F}(\varphi)-\mathcal{F}(P_k\varphi),\mathcal{F}(P_k\psi)
    \rangle_{L^2(\mathbb{R}^n)}\nonumber\\
    &=&\langle \chi_{\{\xi\in\mathbb{R}^n:|\xi|^s-c>k\}}\mathcal{F}(\varphi),
    \chi_{\{\xi\in\mathbb{R}^n:|\xi|^s-c\leq k\}}\mathcal{F}(\psi)\rangle_{L^2(\mathbb{R}^n)}=0.
\end{eqnarray*}
    This, together with the arbitrariness of $\varphi, \psi$, means that $P_k(I-P_k)=0$ for each $k\in\mathbb{N}$. Thus, $\{P_k\}_{k\in \mathbb{N}}$ are projections.
    We next show $\{P_k\}_{k\in \mathbb{N}}$ are orthogonal. Indeed, since
    $\{P_k\}_{k\in \mathbb{N}}$ are projections and self-adjoint, we have that, for each $k\in \mathbb{N}$,
    $\mbox{Range}(P_k)\perp \mbox{Range}(I-P_k)$.
    Because of $P_k$  being projection, it is clear that $\mbox{Range}(I-P_k)=\mbox{Ker}(P_k)$
    for each $k\in\mathbb{N}$. (Indeed, if $f\in \mbox{Ker}(P_k)$, then
    $f=P_kf+(I-P_k)f=(I-P_k)f$. Thus, $f\in \mbox{Range}(I-P_k)$ and then $\mbox{Ker}(P_k)\subset
    \mbox{Range}(I-P_k)$. Conversely, if $f\in \mbox{Range}(I-P_k)$, i.e., there exists a $g\in L^2(\mathbb{R}^n)$ such that
    $f=(I-P_k)g$, then, by the fact $P_k(I-P_k)=0$, we have $P_kf=P_k(I-P_k)g=0$, i.e., $f\in
    \mbox{Ker}(P_k)$. Thus, $\mbox{Range}(I-P_k)\subset\mbox{Ker}(P_k)$. In summary, we can conclude that $\mbox{Range}(I-P_k)=\mbox{Ker}(P_k)$.)
    Therefore, $\mbox{Range}(P_k)\perp \mbox{Ker}(P_k)$ for each $k\in\mathbb{N}$, i.e.,
    $\{P_k\}_{k\in \mathbb{N}}$ are orthogonal.

  Secondly, by \cite[Lemma 3.1 and the inequality (4.1)]{Huang-Wang-Wang}, there exists a $C>0$ such that, for each $k\in\mathbb{N}$,
\begin{equation*}\label{Liu-11-17-01}
    \|P_k\varphi\|_{L^2(\mathbb{R}^n)}\leq e^{Ck^{\frac{1}{s}}}\|B^*P_k\varphi\|_{L^2(\mathbb{R}^n)}\;\;\mbox{for any}\;\;\varphi\in L^2(\mathbb{R}^n);
\end{equation*}
\begin{equation*}\label{Liu-11-17-02}
    \|(I-P_k)S(t)^*\varphi\|_{L^2(\mathbb{R}^n)}\leq e^{-k t}\|\varphi\|_{L^2(\mathbb{R}^n)}\;\;\mbox{for all}\;\;t\in\mathbb{R}^+\ \mbox{and}\;\;\varphi\in L^2(\mathbb{R}^n).
\end{equation*}
The above two inequalities clearly imply the conditions $(b_1)$  and $(a)$  in Theorem \ref{yu-theorem-10-9-1} respectively. Then,  by Theorem \ref{yu-theorem-10-9-1}, we have $(iii)$ of Theorem \ref{yu-theorem-3-4-1}.

Finally, by Theorem \ref{yu-theorem-3-4-1}, we see
 that the equation \eqref{p01} (where $E$ is replaced by $\hat E$) is completely stabilizable.
 \end{proof}

\vskip 5pt
 \noindent \emph{\textbf{Example 2. (Heat equation with Hermite operator in $\mathbb{R}^n$.)}}
 Let $c\geq n$ and let $E$ be a  subset of positive measure in a half-space of  $\mathbb{R}^n$.
We consider the equation:
\begin{equation}\label{p02}
\begin{cases}
\partial_t y(t,x)-\triangle y(t,x)+|x|^2y(t,x)-c y(t,x)=\chi_E(x)u(t,x), &\mbox{in}\;\;\mathbb{R}^+\times \mathbb{R}^n,\\
y(0, \cdot)=y_0(\cdot)\in L^2(\mathbb{R}^n),
\end{cases}
\end{equation}
where  $u\in L^2(\mathbb{R}^+;L^2(\mathbb{R}^n))$.

For the equation \eqref{p02}, we have the following conclusions:
\begin{itemize}
\item The equation \eqref{p02} can be put into the framework in Section \ref{yu-section-1} by the following manner:
Let $X=U:=L^2(\mathbb{R}^n)$ and $A:=\triangle-|x|^2+c$, with
$D(A)=\{f\in L^2(\mathbb{R}^n) : -\triangle f+|x|^2f\in L^2(\mathbb{R}^n)\}$. Let $B: L^2(\mathbb{R}^n)\rightarrow L^2(\mathbb{R}^n)$ be defined in the same way as  that used in  Example 1.
Then $A$ generates a $C_0$-semigroup $\{S(t)\}_{t\geq 0}$ on $X$ (see \cite{Titchmarsh})
and $B\in \mathcal{L}(U;X)$.

\item The equation \eqref{p02} with the null control is unstable
 (see \cite[(1.7), as well as  the note $(e_1)$]{Huang-Wang-Wang}).

 \item It follows by \cite[Theorem 1.10]{Miller} that
the equation \eqref{p02} is not null controllable.

\item  The equation \eqref{p02} is completely stabilizable.  (See Theorem \ref{Liuth4-30} below.)
\end{itemize}

 The next Theorem \ref{Liuth4-30} may have independent interest.

 \begin{theorem}\label{Liuth4-30}
  Let $\hat E$ be a measurable subset in a half-space of $\mathbb{R}^n$. Then the equation  \eqref{p02}
 (where $E$ is replaced by $\hat E$) is completely stabilizable if and only if $\hat E$ has a positive measure.
  \end{theorem}
  \begin{remark}
   Theorem \ref{Liuth4-30} extends \cite[Theorem 1.2]{Huang-Wang-Wang}, which shows that the equation \eqref{p02} (where $E$ is replaced by $\hat E$) is stabilizable if and only if $\hat E$  has a positive measure.
  \end{remark}
\begin{proof}[The proof of Theorem \ref{Liuth4-30}]
 We first show the {\it only if} part. Suppose that the equation \eqref{p02}
(where $E$ is replaced by $\hat E$)
is completely stabilizable, then it is stabilizable. Thus, it follows  by \cite[Theorem 1.2]{Huang-Wang-Wang} that $\hat E$ has a positive measure.

We next show the {\it if} part. Assume that  $\hat E$ has a positive measure. We will use  Theorem \ref{yu-theorem-10-9-1} to show
the statement $(iii)$ of Theorem \ref{yu-theorem-3-4-1}. Indeed, according to  \cite{Titchmarsh},  the operator $A_0:=-\triangle+|x|^2$ has discrete spectrum set $\sigma(A_0)=\{2k+n, k\in \mathbb{N}\}$.
 For each $k\in \mathbb{N}$, we let  $\pi_k$ be the orthogonal projection onto the linear space spanned by the eigenfunctions of $A_0$ associated with the eigenvalue $2k+n$. We then define
the orthogonal projection $P_k$ on $L^2(\mathbb{R}^n)$ by
\begin{equation*}
P_k\varphi:=\sum_{0\leq j\leq (k+c-n)/2}\pi_j\varphi,\;\;\varphi\in \mathbb{R}^n.
\end{equation*}
By \cite[Lemma 3.2 and inequality (4.17)]{Huang-Wang-Wang}, there exists a $C>0$ such that, for each $k\in\mathbb{N}$,
\begin{equation*}\label{Liu-11-17-01}
    \|P_k\varphi\|_{L^2(\mathbb{R}^n)}\leq e^{\frac{n}{2}k\ln k+Ck}\|B^*P_k\varphi\|_{L^2(\mathbb{R}^n)}\;\;\mbox{for any}\;\;\varphi\in L^2(\mathbb{R}^n);
\end{equation*}
\begin{equation*}\label{Liu-11-17-02}
    \|(I-P_k)S(t)^*\varphi\|_{L^2(\mathbb{R}^n)}\leq e^{-k t}\|\varphi\|_{L^2(\mathbb{R}^n)}\;\;\mbox{for all}\;\;t\in\mathbb{R}^+\ \mbox{and}\;\;\varphi\in L^2(\mathbb{R}^n).
\end{equation*}
The above two inequalities imply the conditions $(b_1)$  and $(a)$ (with $\alpha_k=k$) in Theorem \ref{yu-theorem-10-9-1} respectively.
Then,
by Theorem \ref{yu-theorem-10-9-1}, we have $(iii)$ of Theorem \ref{yu-theorem-3-4-1}.
Finally, by Theorem \ref{yu-theorem-3-4-1}, we find that the equation \eqref{p02}
(where $E$ is replaced by $\hat E$)
is completely stabilizable.
\end{proof}

\vskip 5pt

\noindent\emph{\textbf{Example 3. (Point-wise control of 1-D heat equation.)}}\ Consider the
  point-wise controlled
 1-D heat equation:
\begin{equation}\label{h01}
\begin{cases}
y_t(t,x)-y_{xx}(t,x)-cy(t,x)=\delta(x-x_0)u(t), &\mbox{in}\;\;\mathbb{R}^+\times (0, 1),\\
y(t,0)=y(t, 1)=0, &\mbox{in}\;\; \mathbb{R}^+,\\
y(0, x)=y_0(x)\in L^2(0,1), &\mbox{in}\;\; (0, 1).
\end{cases}
\end{equation}
Here, $c> \pi^2$,
 $\delta(\cdot-x_0)$ is the usual Dirac function, while $x_0\in (0,1)$ is given in the following manner:

 First of all, given $x\in \mathbb{R}$, we let  $\|x\|:=\inf_{n\in \mathbb{Z}}|x-n|$
 and write $[x]$ for the integer so that $x-1< [x]\leq x$.
 We next
  construct a continued fraction $[a_1, a_2, \cdots]$ by setting
$$
a_1=2 , q_0=0, q_1=1,
$$
 and defining  successively $q_n, a_n, n\geq 2$ as
\begin{equation}\label{Liu-11-19-01}
q_{n+1}=a_nq_n+q_{n-1}, \;\;a_{n+1}=[e^{q_{n+1}^3}]+1, \;\; n\geq 1.
 \end{equation}
 (About the definition of  continued fractions, we refer readers to \cite[Chapter 1]{Cassels}.)
According to   \cite[Chapter 1, Theorem II and Theorem III]{Cassels}, there is an one-to-one correspondence between continued fractions and real numbers in $(0,1)$. We now let  $x_0$ be the real number corresponding to the above
 $[a_1, a_2, \cdots]$ (for the construction of $x_0$, one can refer to the proof of \cite[ Chapter 1, Theorem III]{Cassels}). It is clear that  $x_0$ is an  irrational number
 since the sequence $\{a_n\}_{n\in\mathbb{N}}$ is infinite.

For the equation \eqref{h01} when $x_0$ is chosen as above, we have the following conclusions:
\begin{itemize}
\item The equation \eqref{h01} can be put into the framework in Section \ref{yu-sec-10-18-1}.
(See the proof of Theorem \ref{theorem4.5,11.26} below.)

\item The equation \eqref{h01} with the null control is unstable, since $c>\pi^2$.

\item The equation \eqref{h01} is not null controllable in any time interval. To see this, we first notice that
 the eigenvalues of operator $-(\partial_x^2+c)$ with domain
 $H_0^1(0,1)\cap H^2(0,1)$ are $\lambda_n:=(n\pi)^2-c$ ($n\in\mathbb{N}$); the corresponding normalized eigenfunctions are $\phi_n(x):=\sqrt{2}\sin(n\pi x)$, $x\in(0,1)$ $(n\in\mathbb{N})$.
We next have from \cite{Dolecki} and \cite[(15) in Chapter 1]{Cassels} that
\begin{equation}\label{Liu-11-19-02}
|\phi_n(x_0)|=|\sqrt{2}\sin(n\pi x_0)|\leq \sqrt{2}\pi\|nx_0\|; \;\;\|q_nx_0\|<\frac{1}{q_{n+1}}\;\;\mbox{for all}\;\; n\in \mathbb{N}.
 \end{equation}
We then arbitrarily fix $T>0$ and consider the series $\sum_{n=1}^\infty\frac{\exp(-\lambda_{q_n} T)}{|\phi_{q_n}(x_0)|}$. By \eqref{Liu-11-19-01}, we see that
$\lambda_{q_n}=\pi q_n^2$ and $q_n\rightarrow +\infty$, as $n\rightarrow+\infty$.
These, along with the definition of $a_n$, yield that there is $N:=N(T)>0$ so that when $n>N$, $\exp(-\lambda_{q_n} T)>1/a_n$, which,  together with \eqref{Liu-11-19-02} and the first equation in \eqref{Liu-11-19-01}, implies
\begin{equation}\label{Liu-11-19-03}
\frac{\exp(-\lambda_{q_n} T)}{|\phi_{q_n}(x_0)|}\geq \frac{1}{a_{n}}\cdot \frac{q_{n+1}}{\sqrt{2}\pi} \geq \frac{q_{n}}{\sqrt{2}\pi}
\;\;\mbox{for all}\;\;  n>N.
 \end{equation}
Since $q_n\rightarrow +\infty$, as $n\rightarrow+\infty$, it follows from \eqref{Liu-11-19-03}
that the series $\sum_{n=1}^\infty\frac{exp(-\lambda_{q_n} T)}{|\phi_{q_n}(x_0)|}$ is divergent. Thus, we can
use \cite[Theorem 1]{Dolecki} to see that the equation (\ref{h01}) is not null controllable over $[0,\widehat{T}]$ for any $\widehat{T}<T$.
Since $T>0$ was arbitrarily taken,  the equation (\ref{h01}) is not null controllable in any time interval.

\item The equation \eqref{h01} is completely stabilizable. (See Theorem \ref{theorem4.5,11.26} below.)
\end{itemize}

 The next Theorem \ref{theorem4.5,11.26} may have independent interest.

\begin{theorem}\label{theorem4.5,11.26}
Let $\hat x_0\in (0,1)$. Then the equation \eqref{h01} (where $x_0$ is replaced by  $\hat x_0$) is completely stabilizable if and only if  $\hat x_0$ is irrational.
\end{theorem}
\begin{proof}
We organize the proof in several steps.

\vskip 5pt

\noindent\emph{Step 1.  We put the equation \eqref{h01} (where $x_0$ is replaced by  $\hat x_0$) into the framework in Section \ref{yu-sec-10-18-1}. }

    Let
 $X:=L^2(0,1)$ and $U:=\mathbb{R}$. Let
$$
    Ay:=(\partial_x^2+c)y, \;\; y\in D(A)=H_0^1(0, 1)\cap H^2(0, 1); \  \ Bu=\delta(\cdot-\hat x_0)u, \;\; u\in U.
$$

  First, we will prove that the above $[A,B]$ satisfies the assumptions $(a)$ and $(b)$ in Lemma \ref{Liu-11-18-remark}. For this purpose, we arbitrarily fix $\rho_0>c$.
 The assumption $(a)$ can be checked easily. Indeed, one can directly check that $\rho_0\in \rho(A)\cap \mathbb{R}$ and the operator $A$, with its domain $D(A)$,  generates an analytic semigroup $S(t)$ ($t\geq 0$) on $X$. From these, it follows that  $(a)$ in Lemma \ref{Liu-11-18-remark} is true.  To show $(b)$, we notice the following facts:

 \begin{itemize}
  \item Fact One.  $H_0^{2\gamma}(0, 1)\subset C[0, 1]$ continuously for each $\gamma>1/4$.
  (See  \cite[Chapter1, Theorem 9.8]{Lions-Magenes}.)

  \item Fact Two. $D((\rho_0I-A)^{\gamma})=H_0^{2\gamma}(0, 1)$ for each $1/4<\gamma<3/4$, where
  the norm of $D((\rho_0I-A)^{\gamma})$ is as: $\|z\|_{D((\rho_0I-A)^{\gamma})}=\|(\rho_0I-A)^{\gamma}z\|_X$, $z\in D((\rho_0I-A)^{\gamma})$.
  (See \cite[Chapter 1, Definition 2.1 and Theorem 11.6]{Lions-Magenes}.)
  \item Fact Three. $B\in \mathcal{L}(U;[C[0,1]]')$.
  (This can be directly checked.)
  \item Fact Four. $[D((\rho_0I-A)^{\gamma})]'=X_{-\gamma}$ for each
 $\gamma>0$, where $[D((\rho_0I-A)^{\gamma})]'$ is the dual space with the pivot space $X(=L^2(0,1))$ and $X_{-\gamma}$ is defined in $(b)$ of Lemma \ref{Liu-11-18-remark}. (This follows from \cite[Chapter 2, Section 2.9]{Tucsnak-Weiss} or \cite[Chapter 1, Theorem 6.2 and Theorem 12.2]{Lions-Magenes}.)
 \end{itemize}
 From these facts, we obtain that
 $B\in \mathcal{L}(U;X_{-\gamma})$ for each
 $1/4<\gamma<3/4$, which leads to $(b)$  in Lemma \ref{Liu-11-18-remark}.
  In summary, the assumptions $(a)$ and $(b)$  in Lemma \ref{Liu-11-18-remark}
  hold for the above $[A,B]$.

Next, we can use Lemma \ref{Liu-11-18-remark} to see that  the assumptions $(\widetilde{H_1})$-$(\widetilde{H_3})$ in Section \ref{yu-sec-10-18-1} are satisfied by $[A,B]$ in the current case. Consequently, the equation \eqref{h01}
has been  put into the framework in Section \ref{yu-sec-10-18-1} by the above way.
\par

\vskip 5pt
\noindent\emph{Step 2.   We prove the sufficiency.}
\vskip 5pt

Assume that  $\hat x_0$  is irrational. Then we have
\begin{equation}\label{yu-10-19-12}
    \phi_n(\hat x_0)\neq 0\;\; \mbox{for all}\;\;n\in\mathbb{N}.
\end{equation}
Here $\phi_n$ is the normalized eigenfunctions defined above (\ref{Liu-11-19-02}). We will verify that all the assumptions in Theorem \ref{yu-corollary-10-10-1}
are true for the current case. {\it When this is done, we can apply
 Theorem \ref{yu-corollary-10-10-1} to get $(iii)$ in Theorem \ref{yu-theorem-10-5-1},
 and then
 use Theorem \ref{yu-theorem-10-5-1} to see that the  equation (\ref{h01}) is completely stabilizable.}

To this end, we first recall that the assumptions $(a)$ and $(b)$ in Lemma \ref{Liu-11-18-remark} have been
checked in Step 1.
We then define, for each $k\in\mathbb{N}$,
\begin{equation*}\label{yu-10-19-13}
    P_k \varphi:=\sum_{i=1}^k\langle \varphi, \phi_i\rangle_{X}\phi_i,\;\; \varphi\in X.
\end{equation*}
  It is clear that, for each $k\in \mathbb{N}$, $P_k$ is the orthogonal projections of $X$ onto the linear span $\mathcal{S}_k:=\mathrm{span}\{\phi_i: i=1, 2, \ldots, k\}$, and
    $P_kS(\cdot)=S(\cdot)P_k$.
From these, we see that the assumption \eqref{yu-10-10-41} in Theorem \ref{yu-corollary-10-10-1} holds.
Next, since  $S(t)^*\phi_n=e^{-\lambda_nt}\phi_n$ for all $n\in\mathbb{N}$ and  $t\in\mathbb{R}^+$, one can  directly check that for each $k\in \mathbb{N}^+$,
\begin{equation*}
    \|(I-P_k)S(t)^*\phi\|_X\leq e^{-\lambda_kt}\|\phi\|_X\;\;
    \mbox{for all}\;\;\phi\in X,\;\;t\in\mathbb{R}^+.
\end{equation*}
    This, together with the definition of $\lambda_k$, implies that the dissipative condition $(a)$ in Theorem \ref{yu-theorem-10-9-1}.
     We finally  show the assumption $(b'_2)$ in Theorem \ref{yu-corollary-10-10-1}. For this purpose, we arbitrarily fix $T_0>0$.
    Define, for each $k\in\mathbb{N}$, the function $p_k(t):=e^{-\lambda_kt}$, $t\in (0,T_0)$. Let  $E(n, T_0)$,
    with $n\in\mathbb{N}$, be the subspace (in $L^2(0,T_0)$),  spanned by the functions $\{p_k\}_{k\in\mathbb{N}\setminus\{n\}}$. Let $d_n$ be the distance of $p_n$ to $E(n, T_0)$ in $L^2(0,T_0)$. Then there are $K>0$ and $\varepsilon>0$ which are independent of $n$ such that (see \cite[Theorem 1.1]{Fattorini})
 \begin{eqnarray*}\label{h04}
 d_n\geq K\exp (-\varepsilon \lambda_n)\;\;\mbox{for all}\;\; n\in\mathbb{N}.
 \end{eqnarray*}
    Thus, it follows by (\ref{yu-10-19-12}) that when $k\in\mathbb{N}$ and  $\varphi=\sum_{n=1}^{+\infty}a_n\phi_n\in D(A^*)$,
 \begin{eqnarray*}\label{yu-10-19-16}
    \int_0^{T_0}\|B^*P_kS(s)^*\varphi\|_U^2ds&=&\int_0^{T_0}\Big\|e^{-\lambda_js}
    a_j\phi_j(\hat x_0)
    +\sum_{1\leq n\leq k,n\neq j}e^{-\lambda_ns}a_n\phi_n(\hat x_0)\Big\|_U^2ds\nonumber\\
    &=&|a_j|^2|\phi_j(\hat x_0)|^2\int_0^{T_0}\Big\|e^{-\lambda_js}-
    \sum_{1\leq n\leq k,n\neq j}\frac{-a_n\phi_n(\hat x_0)}{a_j\phi_j(\hat x_0)}
    e^{-\lambda_ns}\Big\|^2_Uds\nonumber\\
    &\geq& |a_j|^2|\phi_j(\hat x_0)|^2d_j^2\;\;\mbox{for all}\;\;   j\in\{1,2,\ldots, k\}.
 \end{eqnarray*}
   This, together with (\ref{yu-10-19-12}), gives
 \begin{equation*}\label{yu-10-19-17}
   |a_j|^2\leq \frac{1}{d_j^2|\phi_j(\hat x_0)|^2}\int_0^{T_0}\|B^*P_kS^*(s)\varphi\|_U^2ds \;\;\mbox{for all}\;\; j\in\{1,2,\ldots, k\}.
 \end{equation*}
    From the above, we see  that when $k\in\mathbb{N}$,
 \begin{equation*}\label{yu-10-19-18}
    \|P_kS(T_0)^*\varphi\|^2_X=\sum_{j=1}^{k}|a_j|^2e^{-2\lambda_j T_0}
    \leq \sum_{j=1}^{k}\left(\frac{e^{-2\lambda_j T_0}}{d_j^2|\phi_j(\hat x_0)|^2}\right)\int_0^{T_0}\|B^*P_kS(s)^*\varphi\|_U^2ds
    \;\;\mbox{for all}\;\; \varphi\in D(A^*).
 \end{equation*}
   This leads to  the condition $(b'_2)$ (in Theorem \ref{yu-corollary-10-10-1})
    with  $C(k,T_0):=\sum_{j=1}^{k}\left(\frac{e^{-2\lambda_j T_0}}{d_j^2|\phi_j(\hat x_0)|^2}\right)$.

    Hence, all assumptions in Theorem \ref{yu-corollary-10-10-1} are satisfied for the current case.

    \vskip 5pt

\noindent \emph{Step 2.   We prove the necessity.}
\vskip 5pt
By contradiction,  we suppose that (\ref{h01}) (where $x_0$ is replaced by $\hat x_0$) is completely stabilizable,  but
     $\hat x_0$ is a rational number.
          Then there is  $n_0\in \mathbb{N}$ such that
           $\phi_{n_0}(\hat x_0)=0$.
          Write $\varphi(x):=\phi_{n_0}(x)$, $x\in (0,1)$. Arbitrarily fix $T>0$. Then we have
          that $\|\varphi\|_X=1$;
 \begin{eqnarray}\label{h07}
 &&\|S(T)^*\varphi\|_X=\|e^{-\lambda_{n_0}T}\phi_{n_0}\|_X=e^{-\lambda_{n_0}T};\nonumber\\
   &&\| B^*S(T-t)^*\varphi\|_{L^2(0,T; U)}=\left(\int_0^T|e^{-\lambda_{n_0}t}\phi_{n_0}(\hat x_0)|^2dt\right)^{1/2}=0.
 \end{eqnarray}
 Since the equation (\ref{h01}) (where $x_0$ is replaced by $\hat x_0$) is completely stabilizable,  we
obtain from $(iii)$ of Theorem \ref{yu-theorem-10-5-1} (see   (\ref{yu-6-22-1})) and  (\ref{h07}) that for each $\alpha>0$, there exists $C(\alpha)>0$, which is independent on $T$, such that
 \begin{eqnarray*}\label{h08}
 e^{-\lambda_{n_0}T}\leq C(\alpha)e^{-\alpha T} \;\;\mbox{for all}\;\; T>0,
 \end{eqnarray*}
 which is equivalent to
  \begin{eqnarray}\label{h09}
 C(\alpha)\geq e^{(\alpha-\lambda_{n_0})T}\;\;\mbox{for all}\;\; T>0.
 \end{eqnarray}
 However, if we take $\alpha=\lambda_{n_0}+1$, then there is no $C(\alpha)>0$  so that
 (\ref{h09}) is true,  because the right hand side tends to $+\infty$ as $T$ tends to infinity. This leads to a contradiction.
 Thus,  $\hat x_0$ must be irrational.
    \end{proof}

\vskip 5pt
\noindent\emph{\textbf{Example 4. (A periodic controlled system.)}}\ Let $X=U:=l^2$ and $\mathbb{T}:=1$.
Define a sequence $\{\tau_n\}_{n\in\mathbb{N}}$ in the manner: $\tau_n :=\frac{1}{\alpha}\sum\limits_{k=n+1}^\infty a_k$, where $a_k:=e^{-k^2}$ and
$\alpha:=\sum\limits_{k=1}^\infty a_k$. (It is clear that $\alpha\in (0,1)$ and $\tau_n\in (0,1)$ for each $n\in \mathbb{N}$). Let
 \begin{multline*}
  A:=-\mathrm{diag}\{ 1,\,2,\,\cdots,\, n,\,\cdots\};\;\; B(t):=\mathrm{diag}\left\{ \chi_{(\tau_1,\tau_0)}(\{t\}), \;\cdots,  \;\chi_{(\tau_n,\tau_{n-1})}(\{t\}),\;\cdots\right\},\;t\in(0,+\infty),
  \end{multline*}
  where $\{t\}$ denotes  the  fractional  part of $t$, i.e., $\{t\}=t-[t]$, where $[t]$ is the integer so that $t-1<[t]\leq t$.
 Consider the following  $1$-periodic  system:
 \begin{equation}\label{pe1}
\frac{d}{dt}y(t):= \frac{d}{dt}\begin{pmatrix} y_1\\y_2\\ \vdots\\ y_n\\ \vdots\end{pmatrix}(t)=
 A\begin{pmatrix} y_1\\y_2\\ \vdots\\ y_n\\ \vdots\end{pmatrix}(t)
 +
 B(t)\begin{pmatrix} u_1\\u_2\\ \vdots\\ u_n\\ \vdots\end{pmatrix}(t),
 \end{equation}
where $u=(u_1,u_2, \cdots)^\top$ is taken from  $L^2(\mathbb{R}^+; l^2)$.
For the equation \eqref{pe1}, we have the  conclusions:
\begin{itemize}
\item One can directly check that the equation \eqref{pe1} can be put into the framework in Subsection \ref{yu-per-1}.
\item The equation \eqref{pe1} is not null controllable (see Theorem \ref{prop4.6,11.26} given later).
\item The equation \eqref{pe1} is periodically completely stabilizable (see Theorem \ref{prop4.6,11.26} given later).
\end{itemize}

\begin{theorem}\label{prop4.6,11.26}
The system \eqref{pe1} is not null controllable but completely stabilizable.
\end{theorem}
\begin{proof}
We organize the proof in  two steps.

\vskip 5pt

\noindent{\it Step 1. We show that \eqref{pe1} is not null controllable.}

\vskip 5pt

We only need  to prove the following \emph{Statement A}:  For each $m\in\mathbb{N}$, the system \eqref{pe1} is not null controllable over $[0,m]$. To this end, we arbitrarily fix $m\in\mathbb{N}$.
Write $\varphi_m(\cdot;\psi)$ for the solution to the dual system:
 \begin{equation*}\label{pe2}
 \begin{cases}
 \displaystyle\frac{d}{dt}\varphi_m(t)\triangleq\frac{d}{dt}\begin{pmatrix} \varphi_{m,1}\\\varphi_{m,2}\\ \vdots\\ \varphi_{m,n}\\ \vdots\end{pmatrix}(t)=
 \begin{pmatrix} 1\\&2\\ &&\ddots\\&&& n\\ &&&&\ddots\end{pmatrix}\begin{pmatrix} \varphi_{m,1}\\\varphi_{m,2}\\ \vdots\\ \varphi_{m,n}\\ \vdots\end{pmatrix}(t),\;\;t\in [0,m],\\
 \varphi_m(m)=\psi=(\psi_1,\psi_2, \cdots)^\top.
 \end{cases}
 \end{equation*}
Because of the equivalence between the controllability and the observability,
\eqref{pe1} is not null controllable  over $[0,m]$
 if and only if for any $C>1$, there is  $\psi\in l^2$ so that
  \begin{equation}\label{pe3}
 \|\varphi_m(0;\psi)\|_{X}>C\|B^*(\cdot)\varphi_m(\cdot;\psi)\|_{L^2(0, m;U)}.
 \end{equation}
To show \eqref{pe3}, we arbitrarily fix $C>1$. Take $n=n(C)\in \mathbb{N}$ such that
\begin{equation}\label{pe4}
n\geq m+\sqrt{m^2+2\ln C+\ln\frac{2}{\alpha}}.
 \end{equation}
 Then we take $\psi=(\psi_1,\psi_2, \cdots)^\top\in l^2$ with $\psi_n=1$ and $\psi_k=0$, when $k\neq n$.
 By a direct calculation, we find that  $ \|\varphi_m(0;\psi)\|=e^{-nm}$ and that
 $$
\|B^*(\cdot)\varphi_m(\cdot;\psi)\|_{L^2(0, m;U)}^2=\int^{\tau_{n-1}}_{\tau_n}e^{-2n\tau}\mathrm d\tau\sum\limits_{k=0}^{m-1}e^{-2nk}\leq \frac{e^{-n^2}}{\alpha(1-e^{-2n})}<\frac{2}{\alpha}e^{-n^2}.
 $$
 These, together with \eqref{pe4}, leads to \eqref{pe3}.
 Therefore, the system \eqref{pe1} is not null controllable.
 \vskip 5pt

\noindent{\it Step 2. We show that \eqref{pe1} is completely stabilizable.}

\vskip 5pt
 It is sufficient to prove $(ii)$ of Theorem  \ref{maintheorem1}. To this end, we arbitrarily fix
 $k\in\mathbb{N}$. Let $n_k=1$ and
 $ C(k)=\sqrt{\alpha}e^{k^2/2}$. Then we have that for each  $\psi=(\psi_1,\psi_2, \cdots)^\top\in l^2$,
 $$
\|\varphi_{1}(0;\psi)\|^2_{X}=\sum\limits_{n=1}^{k}e^{-2n}\psi_n^2+\sum\limits_{n=k    +1}^{\infty}e^{-2n}\psi_n^2\leq
\sum\limits_{n=1}^{k}e^{-2n}\psi_n^2+e^{-2k}\sum\limits_{n=k+1}^{\infty}\psi_n^2
\leq \sum\limits_{n=1}^{k}e^{-2n}\psi_n^2+e^{-2k}\|\psi\|^2_{X};
$$
 $$
\|B^*(\cdot)\varphi_{1}(\cdot;\psi)\|_{L^2(0,1;U)}^2=\sum\limits_{n=1}^{\infty}\int^{\tau_{n-1}}_{\tau_n}e^{-2n\tau}\mathrm d\tau\psi_n^2
 \geq \sum\limits_{n=1}^{k}\frac{a_n}{\alpha} e^{-2n}\psi_n^2.
 $$
 These, along with the fact that $C(k)^2a_n=e^{k^2}a_n\geq 1$ when $1\leq n\leq k$, yield
 \begin{equation*}\label{pe5}
\|\varphi_{1}(0;\psi)\|_{X}\leq
C(k)\|B^*(\cdot)\varphi_{1}(\cdot;\psi)\|_{L^2(0,1;U)}+e^{-k }\|\psi\|_{X}\;\;\mbox{for any}\;\;\psi\in X,
\end{equation*}
 which, along with Remark \ref{remark3.11,12.8},
  leads to the statement $(ii)$ of Theorem \ref{maintheorem1}. Then it follows from Theorem \ref{maintheorem1} that the system \eqref{pe1} is periodically completely stabilizable.
\end{proof}
\vskip 5pt

\textbf{Acknowledgments.} The authors would like to thank the anonymous referees for the valuable suggestions.

\end{document}